\documentclass[11pt,english]{article}

\usepackage[english]{babel}
\usepackage[utf8]{inputenc}
\usepackage{amsthm}
\usepackage{amsmath}
\usepackage{amssymb} 
\usepackage{graphicx}
\usepackage{dsfont}
\usepackage[bf]{caption}
\usepackage{hyperref} 

\usepackage{anysize}
\marginsize{2cm}{2cm}{0.5cm}{1.5cm}

\newtheorem{thm}{Theorem}[section]

\newtheorem{lem}[thm]{Lemma}

\theoremstyle{definition}
\newtheorem{defi}[thm]{Definition}

\newtheoremstyle{rmk}
  {11pt}                   
  {11pt}                   
  {}                       
  {}                       
  {\normalfont\bfseries}   
  {.}                      
  { }               
  {}

\theoremstyle{rmk}

\newtheorem{rmk}[thm]{Remark}

\newcommand{\N} { \mathbb{N} }
\newcommand{\Z} { \mathbb{Z} }
\newcommand{\R} { \mathbb{R} }

\newcommand{\1}[1]{{\mathds 1}_{\{#1\}}}

\begin{document}

\title{Cutoff and mixing time for transient random walks in random environments}
 \author{}
\date{\today}
\maketitle
\vspace{0.5cm}

\centerline
{Nina Gantert and Thomas Kochler}

\vspace{0.5cm}

\begin{quote}{\small {\bf Abstract: }
We show that a sequence of birth-and-death chains, given by lazy random walks in a transient
environment (RWRE) on $[0,n]$,
exhibits a cutoff in the ballistic regime but does not exhibit a cutoff in the (interior of) the subballistic regime.
We investigate the growth of the mixing times for this model. As an important step in the proof, we derive bounds for the quenched expectation and the quenched variance of the hitting times of the RWRE, which are of independent interest.
}
\end{quote}

\begin{quote}{\small {\bf Key words: }
Cutoff for Markov chains, random walk in random environment, mixing times, hitting times
}
\end{quote}
\begin{quote}{\small {\bf 2010 Mathematics Subject classifications: }
60K37, 60G50, 60J10.
}
\end{quote}
\begin{quote}{\small {\bf Short Title: }
Cutoff and mixing time for transient RWRE
}
\end{quote}

\section{Introduction and statements of the results}
In this paper, we consider a model which is standard by now, namely one-dimensional Random Walk in Random Environement (RWRE). 
The environment $\omega:=(\omega_k)_{k\in\Z}$ is a family of i.i.d.~random variables taking values in $(0,1)$. We denote the distribution of $\omega$ by $\textbf{P}$ and the corresponding expectation by $\textbf{E}$.
After choosing an environment $\omega$ at random according to the law $\textbf{P}$, we define the random walk in random environment (RWRE) as the nearest neighbour random walk $(X_k)_{k\in\N_0}$ on $\Z$ with transition probabilities given by $\omega$: with respect to $P_\omega^z$ ($z\in\Z$), $(X_k)_{k\in\N_0}$ is  the (time homogeneous) Markov chain on $\Z$ with $P_\omega^z(X_0=z)=1$ and
\begin{align}
	P_\omega^z\left[X_{k+1} = i+1\ |\ X_k=i\right]\ &=\ \omega_i, \nonumber\\
	P_\omega^z\left[X_{k+1} = i-1\ |\ X_k=i\right]\ &=\ 1 - \omega_i. \label{RWRE}
\end{align}
for $k\in\N_0$, $i\in\Z$.
However, the question we ask is not a standard question in this context: we fix the environment and consider a sequence of Markov chains, given by
the RWRE with reflection at $0$ and $n$. More precisely, 
for $n\in\N$ the sequence $\omega^n:=(\omega_k^n)_{k\in\{0,..,n\}}$ is given by
\begin{align}
	\omega_k^n :=\ &\begin{cases}
										1 & \text{for } k=0,\\
										\omega_k & \text{for } k= 1,...,n-1,\\
										0				 & \text{for } k=n. \nonumber
								 \end{cases}	
\end{align} 
Now, define $P_{\omega^n}^z$ as the distribution of a RWRE on $\{0,\ldots,n\}$ with reflection in $0$ and $n$ which is, for fixed environement $\omega$, a Markov chain with values in $\{0, \ldots ,n\}$. 
To avoid periodicity problems, we pass to a lazy RWRE which stays in place with probability $1/2$. For fixed environment $\omega$,
$(Y^n_k)_{k\in\N_0}$ is the Markov chain on $\{0,...,n\}$ with $P_{\omega^n}^z(Y^n_0=z)=1$ and with the following transition probabilities: for $i\in \{0,...,n\}$ and $k\in\N$ we have
\begin{align}
	P_{\omega^n}^z\left[Y^n_{k+1} = i+1\ |\ Y^n_k=i\right]\ &=\ \frac{\omega_i^n}2, \nonumber\\
	P_{\omega^n}^z\left[Y^n_{k+1} = i\  |\ Y^n_k=i\right]\ &=\ \frac12, \nonumber\\
	P_{\omega^n}^z\left[Y^n_{k+1} = i-1\ |\ Y^n_k=i\right]\ &=\ \frac{1 - \omega_i^n}2.\label{lRWRE}
\end{align}
For this sequence of Markov chains, we will investigate the behaviour of the mixing times and we ask if the sequence exhibits a cutoff, which roughly means that the distance to equilibrium decays rapidly in a small time window. 
More precisely,
let $(U_k^n)_{k\in\N_0}$ be, for each $n\in\N$, an aperiodic and irreducible Markov chain on a finite state space $\Omega_n$ and let $(\pi_n)_{n\in\N}$ denote the sequence of associated stationary distributions. Further, we assume $$|\Omega_n|\ \stackrel{n\to\infty}{\longrightarrow}\ \infty.$$
\begin{defi}\label{CD1}
 	For the sequence $(U^n_k)_{k\in\N_0}$ the mixing time $t_{\text{mix}}(n)$ is defined by 
 	\begin{align}
 		t_{\textnormal{mix}}(n):=\ &\min\left\{l\in\N : d_n(l)\le \frac14\right\}, \nonumber
 		\intertext{where}
 		d_n(l):=\ &\max_{x\in\Omega_n} \big\Vert \mathbb{P}^x(U_l^n\in \cdot) - \pi_n(\cdot)\big\Vert_{TV} 		\label{CD1.1}
 	\end{align} 	
and $||\cdot||_{TV}$ denotes distance in total variation.
\end{defi}
We note that due to the convergence theorem for aperiodic and irreducible Markov chains we have that $t_{\textnormal{mix}}(n)$ is finite for every fixed $n$ because $d_n(l)\ \stackrel{l\to\infty}{\longrightarrow}\ 0$.  In most relevant cases $t_{\textnormal{mix}}(n)$ tends to infinity with growing state space: the question is, how fast does $t_{\textnormal{mix}}(n)$ grow?
Note that $d_n(k)$ can be interpreted as the worst case distance to stationarity after $k$ steps. \vspace{11pt}\\
Next, we define the cutoff phenomenon for a sequence of aperiodic and irreducible Markov chains. This effect describes a sharp transition of the total variation distance of the distribution of the Markov chain and its stationary distribution from 1 to 0 in a small window around the mixing time. 
\begin{defi}\label{CD2}
	The sequence $(U^n)_{n\in\N}$ exhibits a \textit{cutoff} with cutoff times $(t_n)_{n\in\N}$ and window size $(f_n)_{n\in\N}$ if	
\begin{enumerate}
	\item $\displaystyle f_n = o(t_n),$
	\item	$\displaystyle \lim_{c\to\infty}\liminf_{n\to\infty} d_n(t_n - cf_n)\ =\ 1 \text{ and}$
	\item	$\displaystyle\lim_{c\to\infty}\limsup_{n\to\infty} d_n(t_n + cf_n)\ =\ 0.$
\end{enumerate}
\end{defi} 
As usual, $P_\omega^z$ is called the quenched law of $(X_k)_{k\in\N_0}$ starting from $X_0=z$ and  we denote by $E_\omega^z$ the corresponding quenched expectation. Let $\Z^{\N_0}$ be the space of the paths of the RWRE and let $\mathcal{F}$ be the associated $\sigma$-algebra generated by all cylinder sets. By $\mathbb{P}^z:= \textbf{P}\times P_\omega^z$ we denote the measure on $\left((0,1)^{\Z}\times \Z^{\N_0},  \left(\mathcal{B}_{(0,1)}\right)^{\Z}\otimes \mathcal{F}\right)$ defined by the relation
$$\mathbb{P}^z(B\times F)\ =\ \int_B P_\omega^z(F)\textbf{P}(d\omega),\ \ \ B\in \left(\mathcal{B}_{(0,1)}\right)^{\Z}, F\in \mathcal{F}, $$ where $\mathcal{B}_{(0,1)}$ is the Borel-$\sigma$-algebra on $(0,1)$. The expectation under $\mathbb{P}^z$ is denoted by $\mathbb{E}^z$. We will refer to $\mathbb{P}^z$ and $\mathbb{E}^z$ as the annealed law and the annealed expectation respectively. If $z=0$, we simply write $P_\omega$, $E_\omega$, $\mathbb{P}$ and $\mathbb{E}$. A crucial role will be played by the hitting times of the RWRE, defined by
\begin{equation}\label{T_ndef}
T_n:=\ \min\{l\in\N_0\ :\ X_l = n\}, \,  n =1,2, \ldots
\end{equation}
For $i\in\Z$, let
\begin{align}
\rho_i:=\ \frac{1-\omega_i}{\omega_i}\label{rhodef}
\end{align}
It goes back to \cite{So} that
\begin{equation}\label{positive Geschwindigkeit}
 \mathbb{E} T_1\ < \infty \text{ iff } \textbf{E} \rho_0 < 1. 
\end{equation}
Throughout this work, we make the following assumptions on the environment distribution $\textbf{P}$:\\
\textbf{Assumption 1.} $\textbf{E}\ln\rho_0 < 0$.\\
\textbf{Assumption 2.} There exists a unique $\kappa>0$ such that $$\textbf{E}[\rho_0^{\kappa}]= 1 \text{ and } \textbf{E}[\rho_0^\kappa \ln^+\rho_0] < \infty.$$ 
We sometimes need a further technical assumption which we mention if it is needed: \\
\textbf{Assumption 3.} The distribution of $\ln \rho_0$ is non-lattice with respect to $\textbf{P}$.
\begin{rmk}
\textbf{1)} Assumption 1 implies transience to the right (cf. Theorem 1.7 in \cite{So}). \\
\textbf{2)} The constant in Assumption 2 has a significant impact on the behaviour of the RWRE. If it exists, its value separates the ballistic ($\kappa >1$) from the sub-ballistic ($\kappa \le 1$) regime. By the law of large numbers (cf. Theorem 1.16 in \cite{So}), we have
				\begin{align*}
					\lim_{n\to\infty} \frac{X_n}n\ =\ \lim_{n\to\infty} \frac{n}{T_n}\ =\ \frac{1}{\mathbb{E}T_1}\ =\ v_\text{P}\ \hspace{10pt} \mathbb{P}\text{-}a.s.
				\end{align*} 
and $v_\text{P} >0$ if and only if $\kappa >1$ (cf.~\eqref{positive Geschwindigkeit}). We will also refer to the case $v_\text{P}>0$ as the case with positive linear speed.\\
\textbf{3)} Assumptions 1 and 2 exclude all deterministic environments.\\
\textbf{4)} Note that Assumption 3 is also used in \cite{K}, \cite{PZ} and \cite{G} to show annealed and quenched limit theorems. We refer to \cite{Ko} to prove that our results are also true under the weaker assumption that the union of the support of the distribution of $\ln \rho_0$ and $\{0\}$ is non-lattice. In particular, this weaker assumption includes many environment distributions which consist of just two possible choices for the transition probabilities which all are excluded by Assumption 3.
\end{rmk}
We investigate for which $\kappa>0$ a sequence of lazy RWRE on  $(\{0,...,n\})_{n\in\N}$ exhibits a cutoff. 
We show that although the lazy  RWRE is transient to the right for all $\kappa>0$, we only observe a sharp transition of the distance in total variation to its stationary distribution in the case of positive linear speed ($\kappa>1)$. Let $t_{\text{mix}}^\omega(n)$ denote the mixing time of the lazy RWRE with respect to $P_{\omega^n}$.
\begin{thm}\label{CT2}
	Let Assumptions 1 and 2 hold and assume $\kappa >1$. Then for $\textbf{P}$-almost every environment $\omega$ the sequence of lazy RWRE $(Y_k^n)_{k\in\N_0}$ on $(\{0,...,n\})_{n\in\N}$ exhibits a cutoff with cutoff times
	\begin{align*} 
		t_\omega(n)&:=\ 2E_{\omega^n}(T_n) 
	\intertext{and window size}
		f_\omega(n)&:=\  \sqrt{\textnormal{Var}_{\omega^n}(T_n)}.
	\end{align*}																				
\end{thm}
Note that although $\mathbb{E}(T_n^2) =\infty$ for $\kappa\le2$, we have  $\textnormal{Var}_{\omega^n}(T_n)< \infty$ for  $\textbf{P}$-almost every environment $\omega$ (cf.~Theorem \ref{BT1}) and all $\kappa>0$. We remark that 
$d_n(k)$ 
is decreasing in $k$. 
For simplicity's sake, we do not write integer parts. \vspace{11pt}\\
In the case $\kappa <1$ we show that there is no cutoff:
\begin{thm}\label{CT3}
	Let Assumptions 1-3 hold and assume $\kappa < 1$. Then for $\textbf{P}$-almost every environment $\omega$ the sequence of lazy RWRE $(Y_k^n)_{k\in\N_0}$ on $(\{0,...,n\})_{n\in\N}$ does not exhibit a cutoff.
\end{thm}
To prove that for $\kappa<1$ there is no cutoff under Assumptions 1-3, we show that for $\textbf{P}$-almost every environment $\omega$ the window within which the total variation distance drops from $1$ to $0$ has the same order as the mixing time, and therefore the transition cannot be sharp in the sense of a cutoff. \vspace{11pt}\\
Furthermore, we determine the order of the mixing time: 
\begin{thm}\label{mixing time}
	For $\textbf{P}$-almost every environment $\omega$ we have
\begin{enumerate}
\item[(a)] $\displaystyle \lim_{n\to\infty} \frac{\ln t_{\textnormal{mix}}^\omega(n)}{\ln n}\ =\ \frac1\kappa$\hspace{26pt} for $0<\kappa\le 1$ and 
\item[(b)] $\displaystyle \lim_{n\to\infty} \frac{t_{\textnormal{mix}}^{\omega} (n)}n   \ =\ 2 \mathbb{E} T_1$\hspace{21pt} for $\kappa>1$.
\end{enumerate}
\end{thm}
To prove Theorems \ref{CT2} - \ref{mixing time}, we will need the following bounds for the (quenched) expectation and 
variance of the hitting times, which are of independent interest.
\begin{thm}\label{BT1}
	We have
\begin{enumerate}
	\item[(a)] $\displaystyle \lim_{n\to\infty} \frac{\ln E_{\omega}\left(T_n\right)}{\ln n}\  =\ \max\left\{\frac1\kappa, 1\right\}\ \ \hspace{15pt} \ \rm{\bf P}-a.s.,$
	\item[(b)] $\displaystyle\lim_{n\to\infty} \frac{\ln\textnormal{Var}_{\omega}\left(T_n\right)}{\ln n}\  =\ \max\left\{\frac2\kappa, 1\right\}\ \ \ \ \rm{\bf P}-a.s.$
\end{enumerate}
\end{thm}
Note that because we consider an i.i.d.~environment, the shift $\Theta$ on the product space (given by $\Theta\omega(i) = \omega(i+1)$) is ergodic with respect to $\textbf{P}$. Therefore, Birkhoff's ergodic theorem yields the following stronger statements for the cases in which the annealed expectation and the annealed variance, respectively, exist: \\ 
for $\kappa>1$ we have
\begin{align}
	\lim_{n\to\infty} \frac{E_\omega T_n}{n}\ &=\ \lim_{n\to\infty} \frac1n \sum_{j=0}^{n-1} E_{\Theta^{j}\omega} T_1 \ =\ \mathbb{E} T_1 \ \ \textbf{ P}-a.s. \label{Ergodensatz}
\intertext{ and for $\kappa>2$ we get}
	\lim_{n\to\infty} \frac{\textnormal{Var}_\omega T_n}{n}\ &=\ \lim_{n\to\infty} \frac1n\sum_{j=0}^{n-1} \textnormal{Var}_{\Theta^{j}\omega}(T_1)\ =\ \textbf{E}(\textnormal{Var}_\omega T_1)\ \ \textbf{ P}-a.s.  \label{Ergodensatz2}
\end{align}

There are general cutoff results for birth-and-death-processes, see \cite{DLP}, but they involve the spectral gap and its behaviour is not known for our RWRE sequence. On the other hand, combining our statements with these
general results, one gains information about the spectral gap of RWRE, we refer to \cite{Ko}.\\
The paper is organized as follows. In Section \ref{Preliminaries}, we relate the hitting times of the RWRE to the hitting times of the lazy RWRE. Section \ref{Potential} gives estimates on the environment, using its (well-known) description with the help of the potential and a block decomposition along the lines of \cite{PZ}. This allows to identify a set of ``typical'', ``good'' environments which will be used to analyse the hitting times.
In Section \ref{variance}, we derive Theorem \ref{BT1}.
Section \ref{cutoff} then uses this result to prove  Theorems \ref{CT2} - \ref{mixing time}, establishing the fact that the RWRE exhibits a cutoff in the ballistic regime and does not exhibit a cutoff in (the interior of) the subballistic 
regime.

\section{Preliminaries}\label{Preliminaries} 
Recall \eqref{rhodef} and define, for $i\in\N$,
\begin{align}
W_i:=\ \sum_{j=-\infty}^i \prod_{k=j}^{i} \rho_k. \label{W}
\end{align}
Further, for $i\in\N$,
\begin{align}
	W^0_i:=\ \sum_{j=1}^i \prod_{k=j}^i \rho_k \label{W_0}
\end{align}
Recall \eqref{T_ndef}. One can recursively compute an explicit formula for the quenched expectation of $T_n$ as a function of the environment (cf.~(2.1.14) in \cite{Z}), and if we assume $\textbf{E}\ln\rho_0 < 0$, we have (cf.~\cite{So})
\begin{equation}\label{Expectation}
	E_\omega^i T_{i+1}\ = 1 + 2 W_i\ <\ \infty\ \ \textbf{ P}\text{-}a.s.
\end{equation}
and therefore $E_{\omega^n}^i T_{i+1} = 1 + 2W^0_i$ for $i\in \{1, \ldots n-1\}$. Note that since $(\rho_k)$ is an i.i.d.~sequence,
\begin{equation}
	\mathbb{E} T_1\ = 1 + 2 \textbf{E} W_0 = 1 + \sum_{k=1}^{\infty} \left(\textbf{E} \rho_0\right)^k 
\end{equation}
which implies \eqref{positive Geschwindigkeit}.\\
Let $T_n^Y:=\ \min\{l\in\N_0\ :\ Y^n_l = n\} $ be the first hitting time of position $n$ of the lazy RWRE $(Y^n)$.
The next lemma relates the quenched expectation and quenched variance of $T_n^Y$ with the corresponding quantities of $T_n$.
\begin{lem} \label{Expectation lRWRE}
	We have 
\begin{enumerate}
	\item[(a)]  $\displaystyle E_\omega T_n^Y \ = \ 2 E_\omega T_n$,
	\item[(b)]  $\displaystyle \textnormal{Var}_\omega T_n^Y\ =\ 4\textnormal{Var}_\omega T_n\ + 2 E_\omega T_n.$
\end{enumerate}
\end{lem}
\begin{proof}
For $k\in\N$ we define 
\begin{equation}\label{Zdef}
Z_k : = \1{Y^n_{k-1} \neq Y^n_k}
\end{equation}
and
$$\tau_k:=\ \inf\left\{l\in\N\ :\ \sum_{j=1}^l Z_j = k\right\}$$
and note that $\tau_k$ is negative binomially distributed with parameters $k$ and $\frac12$ (waiting for the $k$th success). Thus, we have $$E_\omega\left[T_n^Y  \Big| T_n =k\right]\ =\ E_\omega \tau_k\ =\ 2k $$
and we get
\begin{align*}
	E_\omega T_n^Y\ =\ E_\omega \left[E_\omega\left[T_n^Y  \Big| T_n \right]\right] \ =\ 2 E_\omega T_n.
\end{align*}
To obtain \textit{(b)}, we first consider
\begin{align*}
	E_\omega\left[\left(T_n^Y\right)^2 \Big| T_n = k\right]\ &=\ E_\omega \Big[\left(\tau_k\right)^2\Big]\ =\ 2k + 4k^2 
\end{align*}
and we therefore get together with \textit{(a)}
\begin{align*}
	\text{Var}_\omega T_n^Y\ &=\ E_\omega \left[E_\omega\left[\left(T_n^Y\right)^2 \Big| T_n \right]\right] -  \Big(E_\omega T_n^Y\Big)^2\	
	=\ 4 \text{Var}_\omega T_n +  2 E_\omega T_n. 
\end{align*}
\end{proof}
\noindent In the same way,
\begin{align}
	 E_{\omega^n} T_n^Y \ = \ 2 E_{\omega^n} T_n\ \ \text{ and }\ \  \textnormal{Var}_{\omega^n} T_n^Y\ =\ 4\textnormal{Var}_{\omega^n} T_n\ + 2 E_{\omega^n} T_n. \label{Expectation lRWRE2}
\end{align}

\section{Environment: decomposition with the potential}\label{Potential}
In this section, we analyse the potential associated with the environment, which was introduced by Sinai in \cite{S}. A good understanding of the potential will be key to analyse the quenched expectation and quenched variance of $T_n$ in the next section. The potential, denoted by $(V(x))_{x\in\Z}$, is a function of the environment $\omega$ defined in the following way:
\begin{align*}
	V(x):= \begin{cases}
							\displaystyle -\sum_{i=x}^{-1} \ln \rho_i &\text{if } x\le -1,\\
							\displaystyle 0\ &\text{if } x=0, \vphantom{\sum_i^k} \\
							\displaystyle\sum_{i=0}^{x-1} \ln \rho_i &\text{if } x\ge 1.
				 \end{cases}
\end{align*}

Due to Assumption 1 ($\textbf{E}\ln\rho_0<0$), the potential is a random walk with negative drift. But there are some ``blocks" of the environment where the potential is increasing. We will see that the height of these increases depends on $\kappa$ and that the RWRE needs most of the time before hitting position $n$ to cross the highest increase of the potential in the interval $[0,n]$. 
Note that for fixed environment, the Markov chain $(X_n)$ is reversible and the conductance of an edge is given by $C(x-1, x) = \exp(-V(x))$, see \cite{FGP}.
To see the influence of the parameter $\kappa$ on the shape of the potential see Figure \ref{potentialkappagross} and Figure \ref{potentialkappaklein} on pages \pageref{potentialkappagross} and \pageref{potentialkappaklein}. 
\begin{figure}[!htbp]
\includegraphics[viewport=95 537 350 755, scale =1.08]{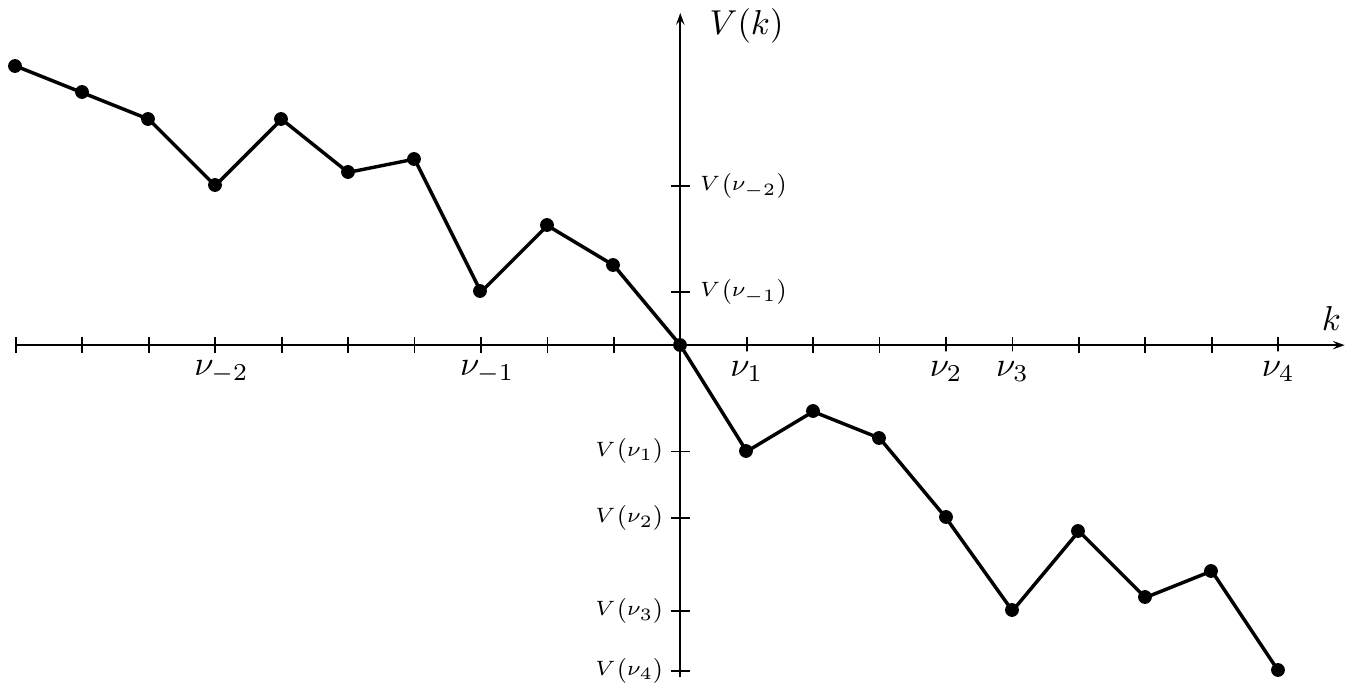}
  \caption{On the definition of the ladder locations.} \label{valley}
\end{figure}
In the following, we use the partition of the potential in blocks introduced in \cite{PZ} and \cite{ESZ09} (cf.~Figure \ref{valley}). We define ``ladder locations" $(\nu_i(\omega))_{i\in\Z}$ of the environment $\omega$ by 
\begin{align}
	\nu_i(\omega):=\ \begin{cases}
								\sup\left\{n < \nu_{i+1}(\omega) : V(n) < V(k) \ \ \forall\ k<n\right\} &\text{for } i \le  -1, \vphantom{\displaystyle \frac12}\\
								 0 &\text{for } i = 0,\vphantom{\displaystyle\frac12}\\
								\inf\left\{n > \nu_{i-1}(\omega) : V(n) < V(\nu_{i-1}(\omega)) \right\} &\text{for } i \ge 1. \vphantom{\displaystyle\frac12}
						\end{cases}	\label{B3}
\end{align} 
If no confusion arises, we often drop the dependence of $\omega$ and simply write $(\nu_i)_{i\in\Z}$. The portion of the environment $[\nu_i,\nu_{i+1})$ is called the $i$th block. Note that the block from $\nu_{-1}$ to $-1$ is different from all the other blocks. Further, we define for $i\in\Z$
\begin{align}
	H_i:=\  \max_{\nu_i\le j<\nu_{i+1}} \left(V(j) - V(\nu_i)\right)\ =\ \max_{\nu_i\le j<k< \nu_{i+1}} \left(V(k) - V(j)\right) \label{H}
\end{align}
as the height of the ith block and for $n\in\N$ let
\begin{align}
	n_0:=\ \max\{l\in\N_0\ :\ \nu_l\le n\} \label{j_0}
\end{align}
denote the number of the block to which $n$ belongs. \\[11pt]
Later, we estimate the height of blocks using results on the asymptotic of the maximum of random walks with negative drift. First, let us assume that the distribution of $\ln \rho_0$ is non-lattice. Then, due to Assumptions 1 and 2, Theorem 1 in \cite{I} yields constants $0<K_1<K_2$ such that for $h>0$ we have
\begin{align}
	K_1\exp(-\kappa\cdot h)\ \le\ \textbf{P}\left(S\ge h\right)\ \le\ K_2\exp(-\kappa\cdot h) \label{PL4.0201},
\end{align}
where $S :=\ \displaystyle \max_{j\ge 0} V(j)$ denotes the maximum of the potential on $\N_0$. Next, let the distribution of $\ln \rho_0$ be concentrated on $x+ y\Z$ for $x\in\R, y\in\R^{>0}$. Therefore, the potential is a Markov chain with i.i.d.~increments of a lattice distribution. Using again Assumptions 1 and 2, we get for the case in which the potential is aperiodic,
due to E19.4 in \cite{Sp} with $r = \exp( \kappa)$ 
\begin{align}
	K'_1\exp(-\kappa\cdot \left( y\cdot n+x\right))\ \le\ \textbf{P}\left(S\ge y\cdot n + x\right)\ \le\ K'_2\exp(-\kappa\cdot\left( y\cdot n+x\right)) \label{PL4.0202},
\end{align}
for $n\in\N$ and constants $0<K'_1<K'_2$. If the potential $V(\cdot)$ is a periodic Markov chain with period $d\in\N$, we still have that $(V(nd+k))_{n\in\N_0}$ is aperiodic for every $k\in\{0,...,d-1\}$ and therefore we get the same asymptotic as in \eqref{PL4.0202} using the minimum and the maximum, respectively, of the appearing constants. Now combining \eqref{PL4.0201} and \eqref{PL4.0202}, we get that under Assumptions 1 and 2 there exist constants $0< \widetilde{C}_1< C_1$ such that we have for all $h>0$
\begin{align}
	\widetilde{C}_1\exp(-\kappa\cdot h)\ \le\ \textbf{P}\left(S\ge h\right)\ \le\ C_1\exp(-\kappa\cdot h) \label{PL4.02}.
\end{align}
In the next lemmata, we identify ``typical" and ``good" subsets of the environment, which simplify calculations in the following.
First, we show that it is very unlikely for the potential to stay at a certain level for a long time because of the negative drift. We define  
\begin{align}
	B_1(n):=\ \Big\{&\nexists\ k\in\N,\ i,j\in\{-n,...,n\}\ :\  j-i \ge k(\ln n)^2,\ V(j) > V(i) - k \ln n\Big\}\label{B1(n)}
\end{align} 
as the set of environments for which on the interval $[-n,n]$ the potential decreases at least by $k\ln n$ every $\lceil k(\ln n)^2\rceil$ steps. In particular, we have for environments $\omega\in B_1(n)$ that all blocks in the interval $[-n,n]$ are smaller than $(\ln n)^2$.
\begin{lem} \label{PL10}
	We have 
	\begin{align*}
		\textbf{P}(B_1(n)^\textnormal{c})\ =\ O\left(n^{-2}\right).
	\end{align*}
\end{lem}
\begin{proof}
First, we note that 
\begin{align}
	&\left\{\exists\ j\ge k(\ln n)^2\   \vphantom{\frac12}:\  V(j) > - k \ln n\right\} \nonumber\\
	\subseteq	 &\left\{ V\big(\lfloor k (\ln n)^2\rfloor\big) > - \left(k+\frac4\kappa\right) \ln n\right\}   \cup \left\{\max_{j\ge k\left(\ln n\right)^2} \left(V(j) - V\big(\lfloor k(\ln n)^2\rfloor \big) \right) > \frac4\kappa \ln n \right\}  \label{PL10.1}
\end{align}
because either $V\big(\lfloor k (\ln n)^2\big)\rfloor$ is larger than  $-\left(k+\frac4\kappa\right) \ln n$ or the potential has to increase more than $\frac4\kappa \ln n$ afterwards. Using \eqref{PL4.02}, we get for arbitrary $l\in \N$
\begin{align*}
	\textbf{P} \left(\max_{j\ge l} \left(V(j) - V(l)\right) > \frac4\kappa \ln n\right)\ =\ O\left(n^ {-4}\right).
\end{align*}
This with \eqref{PL10.1} yields
\begin{align}
		\textbf{P}(B_1(n)^\textnormal{c}) \vphantom{\sum}\
		=\ &\textbf{P}\left(\exists\ k\in\N,\ i,j\in\{-n,...,n\}\ : \ j-i \ge k(\ln n)^2,\ V(j) > V(i) - k \ln n\right)\vphantom{\sum_i^k} \nonumber\\
		\le\ &\sum_{k=1}^{\left\lfloor \frac{2n}{(\ln n)^2}\right\rfloor} 2n \textbf{P}\left(\exists\ j\ge k(\ln n)^2  \ : \ V(j) > - k \ln n\right) \displaybreak[0]\nonumber\\
		\le\ & \sum_{k=1}^{\left\lfloor \frac{2n}{(\ln n)^2}\right\rfloor} 2n \Bigg(\textbf{P} \left(V\Big(\lfloor k (\ln n)^2\rfloor\Big) > - \left(k+\frac4\kappa\right) \ln n\right)\nonumber\\
		&\hspace{60pt} +\textbf{P}\left(\max_{j\ge k\left(\ln n\right)^2} \left(V(j) - V\Big(\lfloor k(\ln n)^2\rfloor\Big)\right) > \frac4\kappa \ln n\right)\Bigg)\nonumber\\
		\le\ & \sum_{k=1}^{\left\lfloor \frac{2n}{(\ln n)^2}\right\rfloor} 2 n \textbf{P} \left(V\Big(\lfloor k (\ln n)^2\rfloor\Big) -  \lfloor k(\ln n)^2\rfloor \textbf{E}\ln \rho_0\ >\  \frac{|\textbf{E} \ln \rho_0|}2 \lfloor k(\ln n)^2\rfloor \right)  +O\left(n^{-2}\right).\vphantom{\sum^k} \label{PL10.2}
\end{align}
Since the potential is a sum of i.i.d.~random variables with some finite positive exponential moments due to Assumption 2 and negative expectation due to Assumption 1, we can apply Cramér's Theorem (cf. Theorem 2.2.3 in \cite{DZ}) to obtain an upper bound for \eqref{PL10.2}, that is
\begin{align*}
	\textbf{P}(B_1(n)^\textnormal{c})\ \le\ & 4 n^2 \exp(-c (\ln n)^2) + O\left(n^{-2}\right)
	\end{align*}
for a constant $c>0$, and this finishes the proof.
\end{proof}
Further, we show that the first $n$ appearing blocks on the right side of $0$ and on the left side of $0$ are 
not too wide.
\begin{lem}\label{PL11}
	We have $$\textbf{P}(B_2^{\textnormal{c}}) = O\left(n^{-2}\right),$$	
	where $$B_2(n):=\ \left\{ -2\bar{\nu}n \le \nu_{-n}, \nu_n \le 2 \bar{\nu}n  \right\} \ \ \text{ with } \bar{\nu}:= \ \textbf{E} \nu_1.$$
\end{lem}
\begin{proof}
	First, we show that $(\nu_i-\nu_{i-1})$ has exponential tails for all $i\in \Z^{\ne 0}$. Again using Cramér's Theorem for the sequence $(\ln \rho_k)_{k\in\Z}$, we get for large $x>0$, $i\in\Z^{\ne 0}$
	\begin{align*}
		\textbf{P} \left(\nu_i - \nu_{i-1}\ >\ x\right)\ &\le\ \textbf{P}\Big(\big| V(\lfloor x\rfloor) - (\lfloor x\rfloor )\textbf{E} \ln\rho_0\big|\ \ge\ |\textbf{E} \ln\rho_0|\lfloor x\rfloor\Big) \vphantom{\frac12} \ \le\ \exp( - c\cdot x) \vphantom{\frac12}
	\end{align*}
for a constant $c>0$. Thus, we have $$\textbf{E} \exp(\widetilde{c} \nu_1)\ <\ \infty \ \ \ \forall\ \widetilde{c} < c$$ 
and therefore we can also apply Cramér's Theorem for the sequence $(\nu_i-\nu_{i-1})_{i\in\Z^{\ne 0}}$ to obtain
	\begin{align}
       \textbf{P} \left( \sum_{i=1}^n (\nu_i-\nu_{i-1}) > 2 \bar{\nu}n\right) +        \textbf{P} \left( \sum_{i=-n+1}^{-1} (\nu_i-\nu_{i-1}) > \frac32 \bar{\nu}n\right)\ =\  O\left(n^{-2}\right). \label{PL11.1}
	\end{align}
Furthermore, we have due to Lemma \ref{PL10} 
$$\textbf{P}\left( \nu_{-1} < -(\ln n)^2)\right)\ \le\ \textbf{P}\left(B_1(n)^{\textnormal{c}}\right)\ =\ O(n^{-2}),$$
and this together with \eqref{PL11.1} finishes the proof.
\end{proof}
Next, we are interested in an upper and a lower lower bound for the highest block in the interval $[-n,n]$. We define
\begin{equation}\label{obenauch}
	B_3(n):=\left\{ \max_{-n\le i\le n} \max_{k\ge i}\left(V(k) - V(i)\right)\ \le\ \frac{1}{\kappa}\Big(\ln n + 2\ln\ln n\Big)\right\}
\end{equation}
and
\begin{equation}\label{jawohl}
B_4(n):=\left\{\max_{-n\le i\le n} \max_{k\ge i} (V(k)- V(i))\ > \frac1\kappa(\ln n - 4\ln\ln n)\right\}.
\end{equation}
\begin{lem}\label{PL3}
	For $\textbf{P}$-almost every environment $\omega$ there exists a $N(\omega)$, such that $\omega \in B_3(n)\cap B_4(n)$ for all $n\ge N(\omega)$.
\end{lem}
 For a proof of \eqref{obenauch},see Lemma 3.4 in \cite{FGP}. For a proof of \eqref{jawohl}, see Lemma 3.5 in \cite{FGP}. 
(In the proof of Lemma 3.5. in \cite{FGP}, an additional integrability assumption (see (1.3) in \cite{FGP}) is used, but
one can give a proof of \eqref{jawohl} based only on \eqref{PL4.02}).
 \vspace{11pt}\\
As a next step, we want to analyse the frequency of the appearance of blocks with certain heights. We therefore define for $0 < a < 1$ (recall \eqref{j_0} for the definition of $n_0$)
\begin{align*}
	D_n(a):=\ \left\{\bigg| \left\{ 0\le i\le n_{0} \ :\ H_i\ge \frac{a}{\kappa} \Big(\ln n + 2\ln\ln n\Big)\right\}\bigg| \ <\ n^{1-a}\right\}.
\end{align*}
In the next lemma, we show that asymptotically we do not have ``too many" high blocks:
\begin{lem}\label{PL4}
 For all $m\in\N$ we have
	\begin{align*}
		\textbf{P}(D(n,m)^c)\ =\ O\left(n^{-2}\right),
	\end{align*}
where 
\begin{align}
D(n,m):=\ \bigcap_{l=1}^{m-1} D_n\left(\frac{l}{m}\right). \label{D(n,m)}
\end{align}
\end{lem}
\begin{proof} 
	Because $m$ is fixed, it is enough to show 
$$\textbf{P}\left(D_n\left(\frac{l}{ m}\right)^\textnormal{c}\right) \ =\ O\left(n^{-2}\right)$$
for arbitrary $l,m\in \N$, $l <m$. \vspace{11pt}\\	
First, we note that the number of blocks with a height of more than $\frac{l}{\kappa m}\Big(\ln n+ 2\ln\ln n\Big)$ can stochastically be dominated by a binomial random variable $B_n^l$ with parameters $n$ and success probability $\textbf{P}\left(S\ge \frac{l}{\kappa m}\Big(\ln n+2\ln\ln n\Big)\right)$, where $S =\ \max_{j\ge 0} V(j)$ denotes the maximum of the potential on $\N_0$. Due to \eqref{PL4.02} we have 
\begin{align}
	\textbf{P}\left(S\ge \frac{l}{ m}\Big(\ln n+2\ln\ln n\Big)\right) \ &\leq \ C_1 \cdot\left(n\cdot (\ln n)^2\right)^{-\frac{l}{ m}}. \label{PL4.1} \vphantom{\left(\frac{l}{m}\right)^\textnormal{c}}
\end{align}
Therefore, we get
\begin{align*}	
	\textbf{P}\left(D_n\left(\frac{l}{m}\right)^\textnormal{c}\right)\ \le\ \textbf{P}\left(B_n^l \ge n^{1-\frac{l}{\kappa m}}\right).
\end{align*}
Now, using the exponential Markov inequality, this together with \eqref{PL4.1} yields:
\begin{align}
	\textbf{P}\left(D_n\left(\frac{l}{m}\right)^\textnormal{c}\right)\ \le\ &\exp\left(-n^{1-\frac{l}{ m}}\right)\cdot  \textbf{E}\exp(B^l_n) \vphantom{\left(\frac{l}{ m}\right)^\textnormal{c}}\nonumber\\
					\le\ &\exp\left(-n^{1-\frac{l}{ m}}\right)\cdot\left(1 + C_1 (e-1)\left(n(\ln n)^2\right)^{-\frac{l}{ m}}\right)^n \vphantom{\left(\frac{l}{ m}\right)^\textnormal{c}}\nonumber\\
					\le\ &\exp\left(C_1 (e-1) (\ln n)^{-\frac{2l}{ m}}n^{1-\frac{l}{ m}}-n^{1-\frac{l}{ m}}\right)\vphantom{\left(\frac{l}{ m}\right)^\textnormal{c}}\nonumber\\
					=\ &O\left(n^{-2}\right), \label{PL4.2} \vphantom{\left(\frac{l}{ m}\right)^\textnormal{c}}
\end{align} 
where we use $1+x\le \exp(x)$ for $x\ge 0$ to obtain the second last line. 
\end{proof}
Further, we define for $0<a<1$
\begin{align*}	
			E_n(a):=\ \Bigg\{ &\nexists\ n\le k< n +(\ln n)^2\ : \\
			&\max_{n\le l < k}(V(k) - V(l)) > \frac{a}\kappa \ln n,\ \max_{k< i<j\le k +(\ln n)^2} (V(j) - V(i)) > \frac{1-\frac{3a}4}\kappa \ln n\ \Bigg\}
\end{align*}
as the set of environments which do not have two ``large" increases of the potential in a ``small" interval after n.
\begin{lem}\label{PL6}	
	For all $0<a<1$ we have 
	\begin{align*}
		\textbf{P}\left(E_n(a)^{\textnormal{c}} \right)\ =\ O\left(n^{-\left(1+ \frac{a}5\right)}\right).
	\end{align*}	
\end{lem}
\begin{proof}
	For environments $\omega \in E_n(a)^{\textnormal{c}}$ we have two ``large" increases of the potential in the interval $[n, n+2(\ln n)^2]$. The first increase is bigger than $\frac{a}\kappa \ln n$ and the second bigger than $\frac{1-\frac{3a}4}\kappa \ln n$. We therefore have 
 \begin{align*}
 		\textbf{P}\left(E_n(a)^{\textnormal{c}} \right)\ \le\  (\ln n)^4 \textbf{P}\left(S\ge \frac{a}\kappa \ln n\right) \textbf{P}\left(S\ge \frac{1-\frac{3a}4}\kappa \ln n\right) \ \le\ (\ln n)^4 C_1^2 n^{-(1+\frac{a}4)}, 
 \end{align*}
where we use \eqref{PL4.02} to obtain the last line.
\end{proof}
Now, we want to use Lemma \ref{PL6} to show that for n large enough in the interval $[0,n]$ we do not find two ``big" increases of the potential in an interval of size $2 (\ln n)^2$. We define
\begin{align*}
	E(n,a):=\ \Bigg\{ &\nexists\ 0\le k\le n\ :\\
	&\max_{k - (\ln n)^2 \le l < k}(V(k) - V(l)) > \frac{a}\kappa \ln n,\ \max_{k< i<j\le k +(\ln n)^2} (V(j) - V(i))> \frac{1-\frac{3a}4}\kappa \ln n\Bigg\}.
\end{align*}
\begin{lem}\label{PL7}	
	For all $0<a<1$ and $\textbf{P}$-almost every environment $\omega$ there exists $N(\omega)$, such that we have $\omega \in E(n,a)$ for all $n\ge N(\omega)$.
\end{lem}
\begin{proof}
Let
\begin{align*}
	N_0(\omega):=\ & \min\left\{ j \ge 0\ :\ \omega\in E_i(a)\ \ \forall\ i>j\right\}
\end{align*}
and let
\begin{align*}
	M(\omega):=\ \max\Bigg\{&\max_{k< r<s\le k +(\ln N_0(\omega))^2} (V(s) - V(r))\ : \\
	& k\in [0,N_0(\omega)] \text{ with}\max_{k - (\ln N_0(\omega))^2<l < k}(V(k) - V(l)) > \frac{a}\kappa \ln N_0(\omega)\Bigg\}
\end{align*}
be the maximal increase of the potential in an interval of size $(\ln N_0(\omega))^2$ after an increase of more than $\frac{a}\kappa \ln N_0(\omega)$ in the interval $[0,N_0(\omega)]$.\vspace{11pt}\\
Due to Lemma \ref{PL6} and the Borel-Cantelli lemma, we have that $N_0(\omega)$ is finite for $\textbf{P}$-almost every environment $\omega$. Now, we take $N(\omega)$ large enough such that $N(\omega)\ge N_0(\omega)$ and
\begin{align*}
\frac{1-\frac{3a}4}\kappa \ln N(\omega)\ \ge M(\omega).
\end{align*}
Then, for $n\ge N(\omega)$ let $K$ be the size of the maximal increase of the potential in an interval of size $(\ln n)^2$ after an increase of more than $\frac{a}\kappa \ln n$ in the interval $[0,n]$. Then, we have
$$K\ \le\ M(\omega)\ \le\ \frac{1-\frac{3a}4}\kappa \ln n$$ if the increase is in the interval $[0,N_0(\omega)]$ or if the increase is in the interval $(N_0(\omega),n]$ we have $$K\ \le\ \frac{1-\frac{3a}4}\kappa \ln n$$
by the definition of $N_0(\omega)$.
\end{proof}
Next, we proof two technical statements which will be useful for calculations in the next Section. First, we define
$$C^-(\omega):= \sum_{j= -\infty}^{-1} \exp(-V(j)),\ \ \ C^+(\omega):= \sum_{j = 0}^{\infty} \exp(V(j))  \text{ and }  D^-(\omega):= \sum_{j = -\infty}^{-1} \exp(-V(j))\left(W_{j} + W_{j}^2\right).$$
\begin{lem} \label{PL0.1}
We have for $\textbf{P}$-almost every environment $\omega$
$$C^-(\omega) + C^+(\omega) + D^-(\omega)\ <\ \infty.$$
\end{lem}
Before we prove the lemma, we note that due to \eqref{Expectation} we have
\begin{align*}
	&E_\omega T_n - E_{\omega^n} T_n\ =\ 2\sum_{j=0}^{n-1}\sum_{i=-\infty}^0 \exp\big(V(j+1) - V(i)\big)\displaybreak[0]\\
	  =\ &2\sum_{j=0}^{n-1} \exp\big(V(j+1)\big)\sum_{i=-\infty}^0 \exp\big(- V(i)\big) \ \le\ 2 \Big(C^-(\omega)+1\Big) C^+(\omega),  
\end{align*}
which, using Lemma \ref{PL0.1}, yields 
\begin{align}
\lim_{n\to\infty} \frac{E_\omega T_n}{E_{\omega^n} T_n}\ =\ 1\label{PL0.1.0.01}
\end{align}
for $\textbf{P}$-almost every environment $\omega$. 
\begin{proof}[Proof of Lemma \ref{PL0.1}]
First, we note that for $\textbf{P}$-almost every environment $\omega$ by the SLLN  there existis $N_1(\omega)$ large enough such that \begin{align}
	V(j)\ &\ge\ \frac{-\textbf{E} \ln\rho_0}2 j\ \text{ for all }\ j \le -N_1(\omega), \label{PL0.1.01}\\
	V(j)\ &\le\ \frac{\textbf{E} \ln\rho_0}2 j\ \text{ for all }\ j \ge N_1(\omega). \nonumber
\end{align}
Due to Assumption 1 ($\textbf{E} \ln\rho_0 < 0$), we therefore get that for $\textbf{P}$-almost every environment $\omega$ we have \begin{align}
	C^-(\omega) + C^+(\omega) \ <\ \infty. \label{PL0.1.02}
\end{align}
Next, we note that Lemma \ref{PL10} and \ref{PL3} together with the Borel-Cantelli lemma yield $$N_2(\omega) := \inf\{k\in\N\ :\ \omega\in B_1(n) \cap B_3(n) \ \forall\ n\ge k\}\ <\ \infty\ \ \textbf{P}\text{-a.s.}$$
We get for $j \le -N_2(\omega)$ 
\begin{align}
	W_j\ &=\ \sum_{i=2j+1}^{j} \exp\big(V(j+1) - V(i)\big) + \sum_{i=-\infty}^{2j} \exp\big(V(j+1) - V(i)\big) \nonumber\\
		&\le\ (-j)  \left(-2 j\right)^{\frac1\kappa}  \left(\ln (-2j)\right)^{\frac2\kappa} + \sum_{i=-\infty}^{2j} \frac1{i^2} \ \le\ (-j)^{2+ \frac1\kappa}, \label{PL0.1.1} 
\end{align}
where we used for the second last inequality that for $\omega \in B_3(-2j)$ the biggest increase of the potential in the interval $[2j,j]$ is smaller than $\frac1\kappa \Big(\ln (-2j) + 2 \ln\ln (-2j)\Big)$.  Further, we used that since $(-i) - (-j) \ge 2(\ln (-i))^2$ for all $i\le 2j$, we have  $V(j) - V(i-1) < - 2 \ln i $ for $i\le 2j$ and $\omega \in B_1(-i)$. \\[11pt]
Next, we notice that for $-N_2(\omega) <j\le 0$ we get
\begin{align}
	W_j\ &= \ \sum_{i=((-N_2(\omega))\wedge 2j) +1}^{j} \exp\big(V(j+1) - V(i)\big)+  \sum_{i=-\infty}^{(-N_2(\omega))\wedge 2j } \exp\big(V(j+1) - V(i)\big) \ \le\  N_2(\omega)^{2+ \frac1\kappa},\label{PL0.1.2} 
\end{align}
where we this time used that by the definition of $N_2(\omega)$ we have $\omega\in B_3(N_2(\omega))$ and $\omega\in B_1(-i)$ for all $i\le (-N_2(\omega))\wedge 2j$. \vspace{11pt}\\
\begin{figure}[ht]
\includegraphics[viewport=-10 480 380 745, scale =0.75]{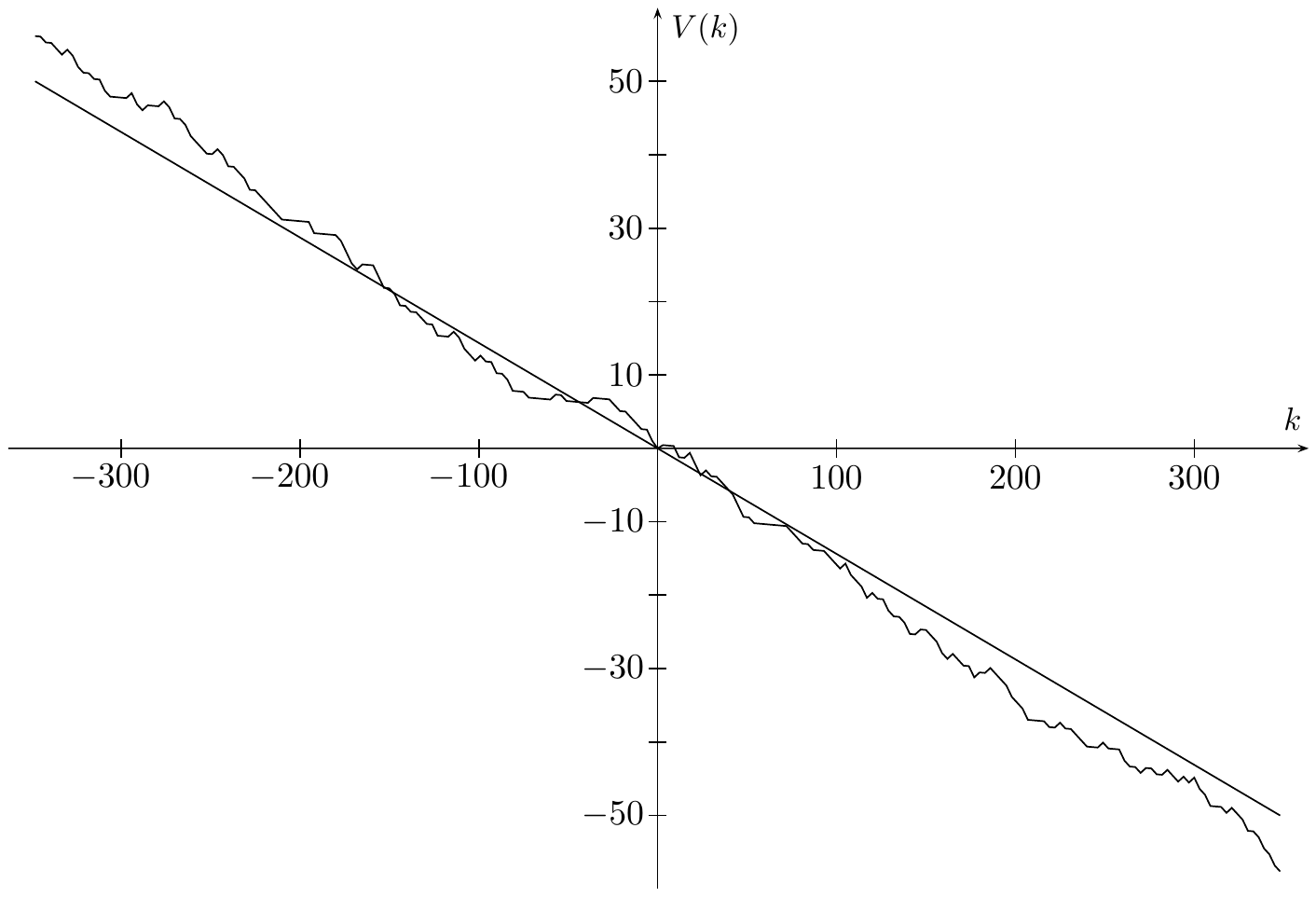}
   \caption{Simulation of the potential of an environment distribution with $\kappa = 2.39$ and $\textbf{E} \ln\rho_0  =  -0.14$. }\label{potentialkappagross}
\end{figure}

Using \eqref{PL0.1.1} and \eqref{PL0.1.2}, we get
\begin{align}
 D^-(\omega)\ \le\ &2\sum_{j = -\infty}^{-1} \exp\big(-V(j+1)\big) W_{j}^2 + 2C^-(\omega)\displaybreak[0]\nonumber \\
	\le\ &2\sum_{j = -\infty}^{-N_2(\omega)} \exp\big(-V(j+1)\big) W_{j}^2 + 2\sum_{j = -N_2(\omega)+1}^{-1} \exp\big(-V(j+1)\big) W_{j}^2+ 2C^-(\omega)\nonumber\\
	\le\ &2\sum_{j = -\infty}^{-N_2(\omega)} \exp\big(-V(j+1)\big)(-j)^{4+ \frac2\kappa} + 2C^-(\omega)\left( N_2(\omega)^{4+ \frac2\kappa}+1\right)\label{PL0.1.3}.
\end{align}
Due to \eqref{PL0.1.01}, we have that the sum in \eqref{PL0.1.3} converges $\textbf{P}$-a.s. and we conclude
$$D^-(\omega) < \infty$$ for $\textbf{P}$-almost every environment $\omega$.
\end{proof}
We define 
\begin{align*}
	C:=\ \left\{\omega\ :\ C^-(\omega) + C^+(\omega) + D^-(\omega)\  <\ \infty\right\}
\end{align*}
and due to Lemma \ref{PL0.1} we have $\textbf{P}(C) =1$. \vspace{11pt}\\
\begin{figure}[ht]
\includegraphics[viewport=-10 480 380 745, scale =0.75]{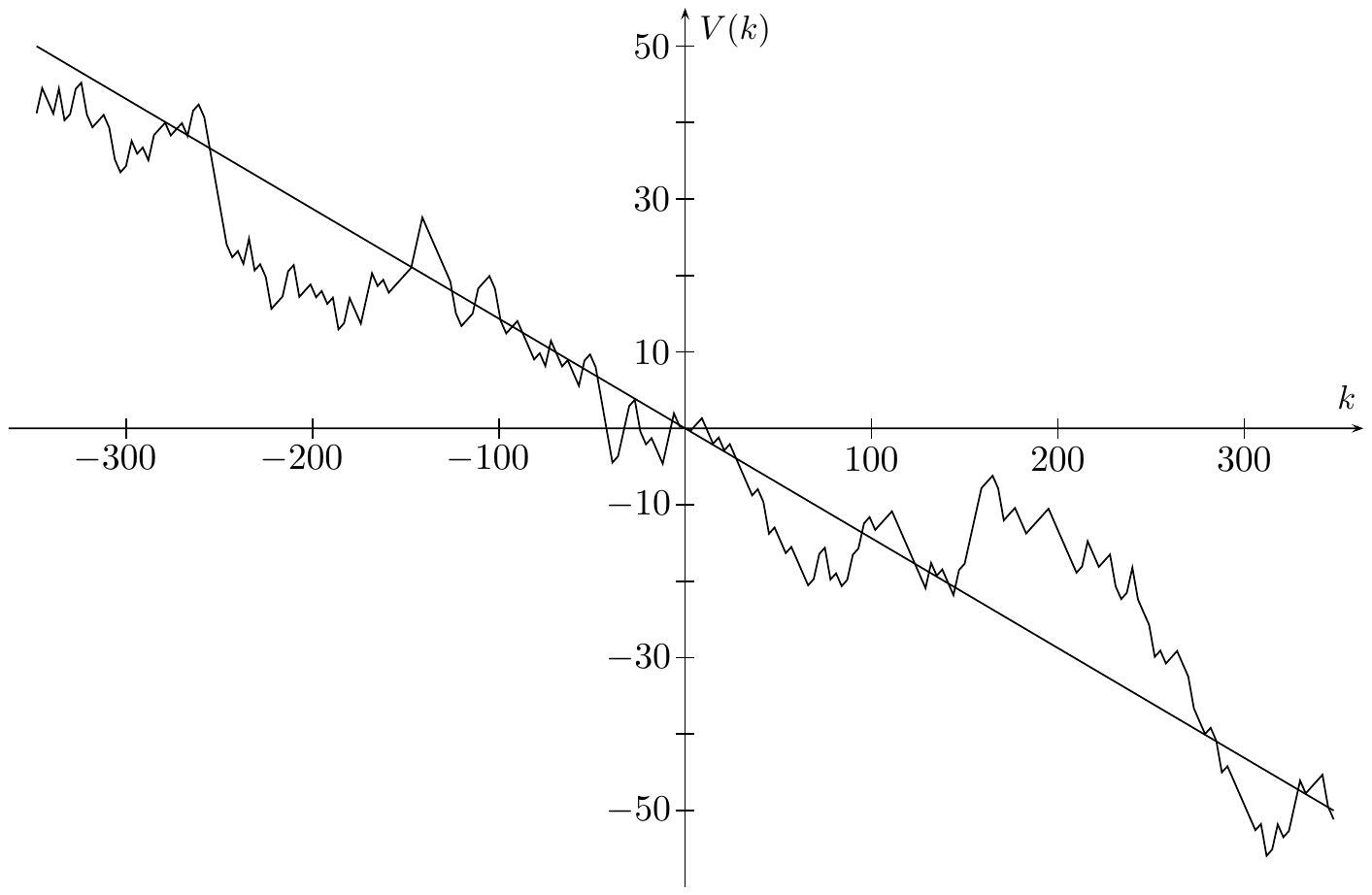}
  \caption{Simulation of the potential of an environment distribution with $\kappa = 0.19$ and $\textbf{E} \ln\rho_0  =  -0.14$.}\label{potentialkappaklein}
\end{figure}
As the last step in this section, we combine all the results of the previous lemmata and define what we will call a ``good" and ``typical" environment:
\begin{lem}\label{PL5}
	For all $m\in\N$, $0<a<1$ and for $\textbf{P}$-almost every environment $\omega$, there exists $N(\omega)$ such that 
	\begin{align*}
			&\omega\in B(n)\cap D(n,m)\cap E(n,a)  \hspace{15pt} \forall\ n \ge N(\omega) ,
	\end{align*}
	where $B(n):=\ B_1(n) \cap B_2(n) \cap B_3(n) \cap B_4(n) \cap C$.

\end{lem}
\begin{proof}
	The statement follows directly from the Borel-Cantelli lemma together with Lemma \ref{PL10} - \ref{PL0.1}.
\end{proof}
For these ``good" and ``typical" environments, we have a lower and an upper bound on the height of the highest block in the interval $[0,n]$ and we know that the potential does not stay at a certain level for a long time. In addition to that, we can control the number of high blocks and do not have two large increases in a small interval. For the influence of $\kappa$ on the shape of the potential compare the following two simulations (Figure \ref{potentialkappagross} and Figure \ref{potentialkappaklein}). We see that the potential in both cases follows the line with slope $-\textbf{E} \ln\rho_0  =  -0.14$ but the fluctuations are larger when $\kappa$ is smaller. 
\section{Hitting times: quenched expectation and quenched variance}\label{variance}
We will see that the distance in total variation of the distribution of the lazy RWRE with respect to $P_{\omega^n}$ to its stationary distribution can be bounded using the tails of $T_n$ (cf.~Section \ref{cutoff}). To show a cutoff, we need to control the fluctuations of $T_n$ very precisely. The main result of this section is Theorem \ref{BT1}. For $\kappa<1$ and for  \textbf{P}-almost every environment, we show that the quenched expectation of $T_n$ is of order $n^{\frac1\kappa}$. Further, we prove that for $\kappa <2$ the quenched variance of $T_n$ is of order $n^{\frac2\kappa}$. We point out that the annealed expectation and variance of $T_n$ are infinite for these $\kappa$. Furthermore, we analyse the quenched variance of the crossing time of ``deep" blocks.
\vspace{11pt}\\
A formula for the quenched variance of $T_{k+1}- T_k$ as a function of the environment is given in equation (2.1) in \cite{G}. In our notation, we get for $k\in\Z$ (cf.~\eqref{W} for the definition of $W_i$)
\begin{align}
	\textnormal{Var}_{\omega}\big(T_{k+1}- T_k\big)\ =\ & 4\left(W_k + W_k^2\right) + 8\sum_{i=-\infty}^{k-1}\exp\big(V(k+1)-V(i+1)\big)\cdot\left(W_i + W_i^2\right). \label{B0.1}
\end{align}
This yields for $n\in\N$
\begin{align}
	\textnormal{Var}_{\omega}\big(T_n\big)\ =\ & 4\cdot\sum_{j=0}^{n-1} \left(W_j + W_j^2\right) + 8\cdot\sum_{j=0}^{n-1}\sum_{i=-\infty}^{j-1}\exp\big(V(j+1)-V(i+1)\big)\cdot\left(W_i + W_i^2\right). \label{B1}
\end{align}
By equations \eqref{B0.1} and \eqref{B1} we see that the quenched variance of $T_n$ is not -- as the expectation -- a function of the sequence $(W_j)_{j\in\N_0}$. But we will see that the  $W_j$-terms in \eqref{B1} still determine the order of $\textnormal{Var}_{\omega}(T_n)$. First, we give an upper bound for $W_j$ depending on the biggest increase of the potential in a small neighbourhood to the left of $j$.
\begin{lem}\label{BL1}
	For $\omega\in B_1(n)\cap C$ and n large enough such that $C^-(\omega) < \ln n$ we have for $j=1,...,n$  (cf.~\eqref{j_0} for the definition of $j_0$)
	\begin{align*}
		W^0_j\ <\ W_j\ \le\  (\ln n)^2  \Big(1 + 2\exp\big(V(j+1) - V(\nu_{j_0})\big)\Big). 
	\end{align*}
\end{lem}
For $\nu_{i+1}<n$ and $n$  large enough, we note that Lemma \ref{BL1} yields for $\omega\in B_1(n)\cap C$ (cf.~\eqref{H} for the definition of $H_i$)
\begin{align}
	\max_{j\in [\nu_i,\nu_{i+1})}W_j\ \le\ (\ln n)^2  +2(\ln n)^2 \exp(H_i). \label{BL1.0.1}
\end{align}
Therefore, we get together with \eqref{Expectation} and using that for $\omega\in B_1(n)$ all blocks in the interval $[-n,n]$ are smaller than $(\ln n)^2$
$$E_\omega^{\nu_{i}} T_{\nu_{i+1}}\ \le\ 2 (\ln n)^4 + 2(\ln n)^4 \exp(H_i).$$
\begin{proof}[Proof of Lemma \ref{BL1}.]
We note that the first inequality of the statement follows directly by the definition. Further, for all $j \in \N_0$ we have by the definition of the ladder locations $(\nu_k)_{k\in\N_0}$ (cf.~\eqref{B3})
	\begin{align*}
		 \max_{0 \le k \le j } \Big(V(j+1) - V(k)\Big) \ =\ V(j+1) - V(\nu_{j_0}). 
	\end{align*}
For $\omega \in B_1(n)\cap C$ we therefore get for $j \ge \nu_1$ and $n$ large enough (cf.~Lemma \ref{PL0.1} for the definition of $C^-(\omega)$)
\begin{align*}
	W_j = &\sum_{k = -\infty}^{j} \exp\big(V(j+1) - V(k)\big) \\
	\le &\sum_{k=-\infty}^{-1} \exp\big(V(j+1) - V(k)\big) + \sum_{l=2}^{\left\lceil \frac{j}{\left\lceil(\ln n)^2\right\rceil}\right\rceil}\hspace{2.3pt}  \sum_{k=j- l \left\lceil(\ln n)^2\right\rceil+2}^{j- (l-1) \left\lceil(\ln n)^2\right\rceil+1}\exp\big(V(j+1) - V(k)\big) \\
	&+ \sum_{k=j- \left\lceil(\ln n)^2\right\rceil+2}^{j}\exp\big(V(j+1) - V(k)\big) \displaybreak[0] \\
			\le & \exp\big(V(j+1)\big)C^-(\omega) + \left\lceil(\ln n)^2\right\rceil\hspace{-5.1pt}\sum_{l=1}^{\left\lceil \frac{j}{\left\lceil(\ln n)^2\right\rceil}\right\rceil-1} \exp\big(-l \ln n\big)  +\hspace{-3pt}\sum_{k= j - \left\lceil(\ln n)^2\right\rceil+2}^{j} \exp\big(V(j+1) - V(k)\big) \\
			\le &2(\ln n)^2\exp\big(V(j+1)- V(\nu_{j_0})\big), \vphantom{\sum_i^k}
\end{align*}
where we used the definition of the set $B_1(n)$ for the second last inequality and the fact that $\exp\big(V(j+1)\big)< \exp\big(V(j+1)- V(\nu_{j_0})\big)$ for $j\ge\nu_1$ to obtain the last inequality. Further, for $j=1,...,\nu_1-1$ and $n$ large enough we have $$\max_{j\in\{1,...,\nu_1-1\}} W_j\ \le\ (\ln n)^2.$$
\end{proof}
\begin{proof}[Proof of Theorem \ref{BT1}]
We have for environments $\omega\in B(n)$ due to \eqref{Expectation} and Lemma \ref{BL1} 
\begin{align}
	E_\omega T_n\ =\ n + 2\sum_{i=0}^{n-1} W_i\ \le\ 3 n (\ln n)^2 + 4 (\ln n)^4 \sum_{l=0}^{n_0} \exp(H_l) \label{BT1.10}
\end{align}
because for $\omega \in B_1(n)$ every block in $[0,n]$ is smaller than $(\ln n)^2$. \vspace{11pt}\\
Next, we recall that for environments $\omega\in D(n,m)$ (cf. \eqref{D(n,m)} for the definition) we have for all  $i\in\{1,...,m-1\}$ at most $n^{1 - \frac{i}m}$ blocks with a height of more than $\frac{i}{\kappa m}(\ln n+ 2\ln\ln n)$. We therefore get for environments $\omega\in B(n)\cap~D(n,m)$
\begin{align}
	\sum_{l=0}^{n_0} \exp(H_l) =\ & \sum_{i=0}^{m-1}\sum_{l=0}^{n_0} \exp(H_l) \mathds{1}_{\left\{\frac{i}{\kappa m}(\ln n+ 2\ln\ln n)\le H_l< \frac{i+1}{\kappa m}(\ln n+ 2\ln\ln n)\right\}}\nonumber\\
	\le\ &\sum_{i=0}^{m-1} \exp\left(\frac{i+1}{\kappa m}(\ln n+ 2\ln\ln n)\right)n^{1-\frac{i}m} \nonumber\\
	\le\ &(\ln n)^{\frac2{\kappa}} \sum_{i=0}^{m-1}n^{\frac{(i+1)+ \kappa(m-i)}{\kappa m}},\label{BT1.9}
\end{align}
where we additionally use that for environments in $B(n)$ the highest block in $[0,n]$ is smaller than $\frac1{\kappa}(\ln n+2\ln\ln n)$ due to definition of $B_3(n)$.\vspace{11pt}\\
Let $\delta >0$ be arbitrary and m be large enough, such that 
\begin{align}
	\frac1m\ <\ \frac\delta2. \label{BT1.4.1}
\end{align}
Further, we note that the function $(i+1) + \kappa(m-i)$ is increasing in $i$ for $\kappa \le 1$ and decreasing for $\kappa >1$. Therefore \eqref{BT1.10} and \eqref{BT1.9} yield for environments  $\omega\in B(n)\cap~D(n,m)$ and n large enough 
\begin{align*}
	E_\omega T_n\ \le\ &3 n (\ln n)^2 + 4 (\ln n)^{4+ \frac2\kappa} \sum_{i=0}^{m-1} n^{\frac{(i+1)+ \kappa(m-i)}{\kappa m}} \\
	\le\ &\begin{cases}
						\displaystyle 3 n (\ln n)^2 + 4 m(\ln n)^{4+ \frac2\kappa} n^{\frac{m+ \kappa}{\kappa m}}\ \ \text{ if }\ \kappa\le 1 \vphantom{\frac12}\\
						\displaystyle 3 n (\ln n)^2 + 4 m(\ln n)^{4+ \frac2\kappa} n^{\frac{1+ \kappa m}{\kappa m}} \ \text{ if }\ \kappa> 1 \vphantom{\frac12}
				\end{cases} 
	\le\ \begin{cases}
						\displaystyle n^{\frac1\kappa + \delta} \ \text{ if }\ \kappa\le 1 \vphantom{\frac12}\\
						\displaystyle n^{1 + \delta} \ \ \text{ if }\ \kappa> 1. \vphantom{\frac12}
				\end{cases}
\end{align*}
Since $\delta>0$ was arbitrary, Lemma \ref{PL5} yields for $\textbf{P}$-almost every environment $\omega$
\begin{align}
	\limsup_{n\to\infty} \frac{\ln E_\omega (T_n)}{\ln n}\ \le\ \max\left\{\frac1\kappa , 1\right\}. \label{BT1.4.0.1}
\end{align}
To obtain the lower bound, we note that $T_n \ge n$ and therefore in the case $\kappa \ge 1$ there is nothing to show. For the case $\kappa <1$ we use that for environments $\omega\in B_4(n)$ we know that in the interval $[0,n]$ there exists a block with a height of more than $\frac1\kappa(\ln n - 4\ln\ln n)$ and we therefore have
\begin{align*}
	E_\omega T_n\ =\ n + 2\sum_{i=0}^{n-1} W_i\ \ge\ (\ln n)^{-\frac4\kappa} n^{\frac1\kappa}.
\end{align*}
Thus, Lemma \ref{PL5} together with \eqref{BT1.4.0.1} finishes the proof of part (\textit{a}). \vspace{11pt}\\
Next we analyse $\text{Var}_\omega T_n$. For $j= 0,..., n$ we define 
\begin{align*}
	H_j^r(n):=\ \max_{j\le k\le j + \lceil(\ln n)^2\rceil} (V(k+1) - V(j+1))
\end{align*}
as the biggest increase of the potential in a neighbourhood of size $(\ln n)^2$ to the right of position $j$. We get for the quenched variance of $T_n$ by changing the order of the summation in equation \eqref{B1} for environments $\omega\in B(n)$ and n large enough
	\begin{align}
		&\textnormal{Var}_{\omega}\left(T_n\right) \vphantom{\frac12}\\
		= &4\sum_{j=0}^{n-1} \left(W_j + W_j^2\right) + 8\sum_{j=0}^{n-1}\sum_{i=-\infty}^{-1}\exp\big(V(j+1) - V(i+1)\big) \left(W_i + W_i^2\right)\nonumber \\ 
		& + 8\sum_{i=0}^{n-2}\sum_{j=i+1}^{n-1}\exp\big(V(j+1) - V(i+1)\big) \left(W_i + W_i^2\right)\displaybreak[0]\nonumber\\
		= &4\sum_{j=0}^{n-1} \left(W_j + W_j^2\right)+ 8 D^-(\omega)\sum_{j=0}^{n-1}\exp(V(j+1))\nonumber\\
		 & + 8\sum_{i=0}^{n-2}\left(\sum_{j=i+1}^{i+\lceil(\ln n)^2\rceil-1}\exp\big(V(j+1) - V(i+1)\big)+\hspace{-6.1pt}\sum_{j=i+\lceil(\ln n)^2\rceil}^{n-1}\exp\big(V(j+1) - V(i+1)\big) \right)\left(W_i + W_i^2\right)\displaybreak[0]\nonumber\\
		\le &4\sum_{j=0}^{n-1} \left(W_j + W_j^2\right)+ 8 D^-(\omega)C^+(\omega)+ 8\sum_{i=0}^{n-2}\left((\ln n)^2\exp\big(H_i^r\big)  + (\ln n)^2 \sum_{k=1}^{\infty}n^{-k}\right)\left(W_i + W_i^2\right)\nonumber\\
		\le & 12(\ln n)^2\sum_{i=0}^{n-1} \left(1 +  \exp\big(H_i^r\big) \right)\left(W_i + W_i^2\right).	\label{BT1.1}
	\end{align}
We notice that for all $i=0,...,n$ we have
\begin{align*}
		(V(i+1) - V(\nu_{i_0})) + H_i^r(n)\ =\ \max_{i\le k < i + \lceil(\ln n)^2\rceil}(V(k+1) - V(\nu_{i_0}))\ 
		\le\  \max_{i_0 \le k\le i_0 + \lceil(\ln n)^2\rceil} H_{k}.  
\end{align*}
Therefore, Lemma \ref{BL1} yields for $\omega\in B(n)$ and $n$ large enough as an upper bound for \eqref{BT1.1} 
\begin{align}
	\textnormal{Var}_{\omega}\left(T_n\right)\	\le\ & 24(\ln n)^6\sum_{i=0}^{n-1}\left(1+\exp\big(H_i^r\big)\right)\Big(2 + 8\exp\big(2\left(V(i+1) - V(\nu_{i_0})\right)\big)\Big)\nonumber\displaybreak[0]\\
	\le\ & 192(\ln n)^6\sum_{i=0}^{n-1}\bigg(\exp\big(H_i^r+ 2\left(V(i+1) - V(\nu_{i_0})\right)\big) +	\exp\big(H_i^r\big)\nonumber\\
	& \hphantom{48(\ln n)^6\sum_{i=0}^{n-1}\Big(} +\exp\big(2\left(V(i+1) - V(\nu_{i_0})\right)\big) \bigg) + 48 (\ln n)^6 n\nonumber\\
	\le\ & 576(\ln n)^{10}\sum_{l=0}^{n_0}\exp\big(2 H_l\big) +48 (\ln n)^6 n.	\label{BT1.3}
\end{align}
Analogously to \eqref{BT1.9}, we get for environments $\omega\in B(n)\cap~D(n,m)$
\begin{align}
	\sum_{l=0}^{n_0} \exp\big(2H_l\big) \le\ & \sum_{i=0}^{m-1}\sum_{l=0}^{n_0} \exp\big(2H_l\big) \mathds{1}_{\left\{\frac{i}{\kappa m}(\ln n+ 2\ln\ln n)\le H_l\le \frac{i+1}{\kappa m}(\ln n+ 2\ln\ln n)\right\}}\nonumber\\
	\le\ &\sum_{i=0}^{m-1} \exp\left(2\frac{i+1}{\kappa m}(\ln n+ 2\ln\ln n)\right)n^{1-\frac{i}m} \ 
	\le\ (\ln n)^{\frac4{\kappa}} \sum_{i=0}^{m-1}n^{\frac{2(i+1)+ \kappa(m-i)}{\kappa m}}.\label{BT1.4}
\end{align}
Let $\delta >0$ be arbitrary and $m$ chosen as in \eqref{BT1.4.1}. Note that the function  $2(i+1) + \kappa(m-i)$ is increasing in $i$ for $\kappa \le 2$ and decreasing for $\kappa >2$. We therefore get as an upper bound for $\text{Var}_\omega T_n$, using equations \eqref{BT1.3} and \eqref{BT1.4} for environments $\omega\in B(n)\cap D(n,m)$, and $n$ large enough
\begin{align}
	\textnormal{Var}_{\omega}\left(T_n\right)\ &\le\ 576 (\ln n)^{10 + \frac4{\kappa}} \sum_{i=0}^{m-1}n^{\frac{2(i+1)+ \kappa(m-i)}{\kappa m}} +48 (\ln n)^6 n\nonumber\\
	&\le\ \begin{cases}
					\displaystyle	576 m(\ln n)^{10 + \frac4{\kappa}} n^{\frac{2m+\kappa}{\kappa m}}+ 48 (\ln n)^6 n\ \text{ if } \kappa\le 2 \vphantom{\frac12}\\
					\displaystyle	576 m(\ln n)^{10 + \frac4{\kappa}} n^{\frac{2+\kappa m}{\kappa m}} + 48 (\ln n)^6 n\ \text{ if } \kappa> 2  \vphantom{\frac12}
				\end{cases}
	\le\ \begin{cases}
						\displaystyle n^{\frac2\kappa +\delta}\ \text{ if } \kappa\le 2 \vphantom{\frac12}\\
						\displaystyle n^{1 +\delta}\ \ \text{ if } \kappa> 2. 		\vphantom{\frac12}			 
				\end{cases} \label{BT1.5}
\end{align}
Since $\delta>0$ was arbitrary, \eqref{BT1.5} together with Lemma \ref{PL5} yields for $\textbf{P}$-almost every environment $\omega$
\begin{align}
	\limsup_{n\to\infty} \frac{\ln \textnormal{Var}_\omega (T_n)}{\ln n}\ \le\ \max\left\{\frac2\kappa , 1\right\}.\label{BT1.6}
\end{align}
Now we turn to the lower bound. We first consider the case $\kappa <2$. For environments $\omega\in B_4(n)$ we have at least one block with a height of more than $\frac1\kappa(\ln n - 4\ln\ln n)$. Together with \eqref{B1} this yields for environments $\omega\in B_4(n)$
\begin{align}
	\textnormal{Var}_\omega (T_n)\ \ge\ \max_{0\le i\le n-1} W_i^2\ \ge\ \exp\left(\frac2\kappa(\ln n - 4\ln\ln n)\right)\ =\  (\ln n)^{-\frac8\kappa}n^{\frac2\kappa}. \label{BT1.7}
\end{align}
For $\kappa \ge2$ we have 
\begin{align*}
		\textnormal{Var}_\omega (T_n)\ \ge\ \sum_{i=0}^{n-1}W_i\ \ge\ \sum_{i=0}^{n-1} \exp(V(i) - V(i-1))\ =\ \sum_{i=0}^{n-1} \rho_i.
\end{align*}
The SLLN then yields that for $\textbf{P}$-almost every environment $\omega$ and n large enough we have
\begin{align}
	\textnormal{Var}_\omega (T_n)\ \ge\ \frac12 n\textbf{E} \rho_0. \label{BT1.8}
\end{align}
Note that due to Jensen's inequality we have $\textbf{E} \rho_0 \le \left(\textbf{E} \rho_0^\kappa \right)^{\frac1\kappa} = 1$ for $\kappa >1$. \vspace{11pt}\\
Lemma \ref{PL5} combined with \eqref{BT1.7} and \eqref{BT1.8} shows that for $\textbf{P}$-almost every environment $\omega$ we have
\begin{align*}
	\liminf_{n\to\infty} \frac{\ln \textnormal{Var}_\omega(T_n)}{\ln n}\ \ge\ \max\left\{\frac2\kappa , 1\right\},
\end{align*}
which together with \eqref{BT1.6} finishes the proof of part (\textit{b}). 
\end{proof}
Next, we define a modified RWRE which we force not to backtrack too far. Let $(\widetilde{X}_k^{(n)})_{k\in\N_0}$ be the random walk which has the same transition probabilities as $(X_k)_{k\in\N_0}$ with the following additional condition: after reaching a new block $\nu_k$ for the first time, the process forms from that time on a random walk in the environment $\widetilde{\omega}^k$, which is defined by 
\begin{align*}
	 \widetilde{\omega}_i^k:= \begin{cases} 1 & \text{for  } i=\nu_{(k - \left\lceil (\ln n)^2\right\rceil)\vee 0},\\
																					\omega_i & \text{else.}
														\end{cases}
\end{align*}
Now, the transition probabilities stay the same until the process reaches the next new block to the right. From that time on, the process forms a random walk in the environment $\widetilde{\omega}^{k+1}$ and so on. Due to this definition, the process $(\widetilde{X}_k^{(n)})_{k\in\N_0}$ cannot backtrack more than $\lceil(\ln n)^2\rceil$ blocks after reaching a new block for the first time. Note that there exists a coupling of the processes $(X_k)_{k\in\N_0}$ and $(\widetilde{X}_k^{(n)})_{k\in\N_0}$ such that we have
$
	\widetilde{X}_k^{(n)} \ge X_k
$
for all $k\in\N_0$ with equality holding until the process $(X_k)_{k\in\N_0}$ backtracks more than 
$\left\lceil (\ln n)^2\right\rceil$ blocks for the first time. For $n\in\N$ we define 
\begin{align}
	\widetilde{T}_n^{(r)}:\ \inf\left\{k\ :\ \widetilde{X}_k^{(r)}= n\right\} \label{widetildeT}
\end{align}
as the first time the restricted process which backtracks not more than $\lceil(\ln r)^2\rceil$ blocks hits position $n$. We further define
\begin{align}
	A(n):=\ \Big\{\ T_n\ =\ \widetilde{T}_n^{(n)}\ \Big\} \label{A(n)}
\end{align}
as the event that the random walk with reflection $(\widetilde{X}^{(n)}_k)_{k\in\N_0}$ reaches position $n$ at the same time as the random walk $(X_k)_{k\in\N_0}$. The next lemma shows that for analysing the distance in total variation of the distribution of the RWRE and its stationary distribution it is sufficient to consider the distribution of $(\widetilde{X}_k^{(n)})_{k\in\N_0}$.
\begin{lem}\label{PL1} 
Define the sequence $\widetilde{\omega}:=(\widetilde{\omega}_k)_{k\in\N}$ by
\begin{align}
	\widetilde{\omega}_k :=\ &\begin{cases}
										1 & \text{for } k=0, \\
										\omega_k & \text{for } k> 0
								 \end{cases}	\label{tilde}
\end{align}
Then, for all $k\in\{1,...,n\}$ we have for $\textbf{P}$-almost every environment $\omega$ 
	\begin{align*}
		\lim_{n\to\infty} P^k_{\widetilde{\omega}}\left(A(n)^c\right)\ =\ 0.
	\end{align*}
\end{lem}
For a proof see Lemma 4.5 in \cite{PZ}.

\begin{lem}\label{BL2}
	For environments 
	\begin{align}
		\omega\in F(\nu_{n+1}):=  B(\nu_{n+1})\cap E\left(\nu_{n+1},\frac23\right),\label{F}
	\end{align}
$n$ large enough, in particular such that $C^-(\omega) < \ln n$ and blocks with 
	\begin{align}
			H_n\ >\ \frac3{4\kappa}\ln(\nu_{n+1})\label{BL2.0.1}	
	\end{align}
	we have 
	\begin{align*}
		\Big(E_\omega^{\nu_{n}}T_{\nu_{n+1}}\Big)^2 \ \le\ 4\textnormal{Var}_\omega(T_{\nu_{n+1}}- T_{\nu_{n}}).
	\end{align*}	
Further, we have $\Big(E_{\widetilde{\omega}}^{\nu_{n}}T_{\nu_{n+1}}\Big)^2  \le 4\textnormal{Var}_{\widetilde{\omega}}(T_{\nu_{n+1}}- T_{\nu_{n}})$ and $\Big(E_\omega^{\nu_{n}}\widetilde{T}_{\nu_{n+1}}^{(n)}\Big)^2  \le 4\textnormal{Var}_\omega\left(\widetilde{T}_{\nu_{n+1}}^{(n)}- \widetilde{T}_{\nu_{n}}^{(n)}\right)$.
\end{lem}
\begin{proof}
	First, we note that due to \eqref{Expectation} we have
	\begin{align*}
	 \Big(E_\omega^{\nu_n}T_{\nu_{n+1}}\Big)^2  \
	=\ &(\nu_{n+1} - \nu_n)^2 + 4(\nu_{n+1} - \nu_n)\sum_{j=\nu_n}^{\nu_{n+1}-1} W_j +4\sum_{j=\nu_n}^{\nu_{n+1}-1} W_j^2 + 8 \sum_{j=\nu_n}^{\nu_{n+1}-2}\sum_{l=j+1}^{\nu_{n+1}-1} W_j W_l.		
	\end{align*}
Together with \eqref{B0.1} this yields	
	\begin{align}
		\Big(E_\omega^{\nu_n}T_{\nu_{n+1}}\Big)^2 - \textnormal{Var}_\omega(T_{\nu_{n+1}}- T_{\nu_{n}}) 
		= & 	\big(\nu_{n+1} - \nu_n\big)^2 + 4\big(\nu_{n+1} - \nu_n-1\big)\sum_{j=\nu_n}^{\nu_{n+1}-1} W_j		\label{BL2.3} \\
		 & + 8\sum_{j=\nu_n}^{\nu_{n+1}-2}\sum_{l=j+1}^{\nu_{n+1}-1}W_j\Big(W_l - \exp\big(V(l+1) - V(j+1)\big)\big(1 + W_j\big)\Big) \label{BL2.4} \\
		 & - 8\sum_{i =-\infty}^{\nu_n -1}\sum_{j=\nu_n}^{\nu_{n+1}-1} \exp\big(V(j+1)-V(i+1)\big)\left(W_i + W_i^2\right).\label{BL2.5}
	\end{align}
Note that \eqref{BL2.5} is negative, since all terms of the sums are positive and therefore can be neglected on our way to find an upper bound. \vspace{11 pt}\\
Next, recall that for $\omega\in B_1(\nu_{n+1})$, all blocks in the interval $[-\nu_{n+1},\nu_{n+1}]$ are smaller than $(\ln \nu_{n+1})^2$, and thus we obtain using  Lemma \ref{BL1} for $\omega\in F(\nu_{n+1})$ the following upper bound for \eqref{BL2.3}
\begin{equation}\label{BL2.6}
(\nu_{n+1} - \nu_n)^2 + 4(\nu_{n+1} - \nu_n-1)\sum_{j=\nu_n}^{\nu_{n+1}-1} W_j\ 
\leq (\ln \nu_{n+1})^4+  4(\ln \nu_{n+1})^6\left(1+  2\exp\left( H_n\right)\right)		 
\end{equation}
On the other hand, note that
\begin{equation}\label{ouf}
\textnormal{Var}_\omega(T_{\nu_{n+1}}- T_{\nu_n})= \textnormal{Var}_{\Theta^{\nu_n} \omega}(T_{\nu_{n+1}- \nu_n})\geq \max\limits_{0 \leq i \leq \nu_{n+1}- \nu_n}W_i^2 \geq \exp\left(2H_n\right)
\end{equation}
where we used \eqref{BT1.7} for the second inequality.
From \eqref{BL2.6}  and \eqref{ouf} we see, taking into account \eqref{BL2.0.1}, that the r.h.s. of \eqref{BL2.3} is bounded above by $\textnormal{Var}_\omega(T_{\nu_{n+1}}- T_{\nu_n})$.
Furthermore, we observe that for $j<l$ we have 
\begin{align*}
	& W_l - \exp\big(V(l+1) - V(j+1)\big)\big(1 + W_j\big) \vphantom{\sum_{k=-\infty}^{l-1}}\\	
	=\ &\sum_{k=-\infty}^{l} \exp\big(V(l+1)-V(k)\big) \\
		&- \exp\big(V(l+1) - V(j+1)\big)\left(1 + \sum_{k=-\infty}^{j}\exp\big(V(j+1)-V(k)\big)\right) \\
	=\ &\begin{cases}
										\displaystyle \sum_{k=j+2}^{l} \exp\big(V(l+1) - V(k)\big) &\text{ if } j<l-1 \\
										\displaystyle 0 &\text{ if } j=l-1.\vphantom{\sum_i^k}
																			 \end{cases}
\end{align*}	
This simplifies \eqref{BL2.4} for environments $\omega \in F(\nu_{n+1})$ to
\begin{align}
				&8\sum_{j=\nu_n}^{\nu_{n+1}-2}\sum_{l=j+1}^{\nu_{n+1}-1}W_j\Big(W_l - \exp\big(V(l+1) - V(j+1)\big)\big(1 + W_j\big)\Big)\nonumber\\
		=\ &8\sum_{j=\nu_n}^{\nu_{n+1}-3}\sum_{l=j+2}^{\nu_{n+1}-1}W_j\sum_{k=j+2}^{l} \exp\big(V(l+1) - V(k)\big)\nonumber \displaybreak[0]\\
		\le\ & 8\ln(\nu_{n+1})^4\exp\big(H_n\big)\sum_{j=\nu_n}^{\nu_{n+1}-3} W_j \mathds{1}_{\left\{W_j \le \frac{\exp\left(H_n\right)}{8\ln(\nu_{n+1})^6}\right\}} \nonumber\\	
				 & + 8\sum_{j=\nu_n}^{\nu_{n+1}-3} W_j \mathds{1}_{\left\{W_j > \frac{\exp\left(H_n\right)}{8\ln(\nu_{n+1})^6}\right\}} 
				 			\sum_{l=j+2}^{\nu_{n+1}-1}\sum_{k=j+2}^{l} \exp\big(V(l+1) - V(k)\big)\nonumber\\
		 \le\ & \exp\big(2H_n\big)  + 8\sum_{j=\nu_n}^{\nu_{n+1}-3} W_j \mathds{1}_{\left\{W_j > \frac{\exp\left(H_n\right)}{8\ln(\nu_{n+1})^6}\right\}}  
				 			\sum_{l=j+2}^{\nu_{n+1}-1}\sum_{k=j+2}^{l} \exp\big(V(l+1) - V(k)\big).\label{BL2.7}
\end{align}
To get an upper bound for \eqref{BL2.4}, we therefore have to control the last two sums in \eqref{BL2.7}. We note that for an environment $\omega\in F(\nu_{n+1})$ and $\nu_n< j< \nu_{n+1}$ with $W_j > \frac{\exp\left(H_n\right)}{8\ln(\nu_{n+1})^6} $ we have 
\begin{align*}
	\exp\big(V(j+1) - V(\nu_n)\big)\ \ge\  \frac{\exp\big(H_n\big)}{17\ln(\nu_{n+1})^8}
\end{align*}
because otherwise
\begin{align*}
	W_j 
	\le \hspace*{-3pt}&\sum_{k=-\infty}^{\nu_n - \lceil(\ln \nu_n)^2\rceil}\hspace*{-3pt} \exp\big(V(j+1)-V(k)\big) + \lceil(\ln \nu_n)^2\rceil + (j+1-\nu_n+ \lceil(\ln \nu_n)^2\rceil )\exp\big(V(j+1)- V(\nu_n)\big)\vphantom{\frac12}\\
	\le & 2(\ln \nu_n)^2 +  2(\ln \nu_n)^2 \frac{\exp\big(H_n\big)}{17\ln(\nu_{n+1})^8}\
	\le\ \frac{\exp\big(H_n\big)}{8\ln(\nu_{n+1})^6}.
\end{align*}
Thus, we get for $\omega\in F(\nu_{n+1})$ using assumption \eqref{BL2.0.1} 
\begin{align}
	V(j+1) - V(\nu_{n})\ \ge\ \frac2{3\kappa} \ln(\nu_{n+1}).  \label{BL2.9.1}
\end{align}
This yields the following upper bound for the summands in \eqref{BL2.7} for environments $\omega \in F(\nu_{n+1})$
\begin{align}
	&8W_j \mathds{1}_{\left\{W_j > \frac{\exp\left(H_n\right)}{8\ln(\nu_{n+1})^6}\right\}} \sum_{l=j+2}^{\nu_{n+1}-1}\sum_{k=j+2}^{l} \exp\big(V(l+1) - V(k)\big) \nonumber\\
	\le\ & 8W_j \mathds{1}_{\left\{W_j > \frac{\exp\left(H_n\right)}{8\ln(\nu_{n+1})^6}\right\}} \ln(\nu_{n+1})^4 \left( \nu_{n+1}\right)^{\frac1{2\kappa}}\vphantom{\sum_i^k}\nonumber\\
	\le\ & W_j^2 \mathds{1}_{\left\{W_j > \frac{\exp\left(H_n\right)}{8\ln(\nu_{n+1})^6}\right\}}, \vphantom{\sum_i^k}\label{BL2.10}
\end{align}
where we used that for $\omega\in E\left(\nu_{n+1},\frac23\right)$ the potential does not increase more than $\frac1{2\kappa}\ln(\nu_{n+1})$ on the interval $\{j,...,\nu_{n+1}\}$ due to \eqref{BL2.9.1}.

For environments $\omega\in F(\nu_{n+1})$, \eqref{BL2.7} and \eqref{BL2.10} together now imply the following upper bound for \eqref{BL2.4} 
\begin{align}
	8\sum_{j=\nu_n}^{\nu_{n+1}-2}\hspace{5pt} \sum_{l=j+1}^{\nu_{n+1}-1}W_j\Big(W_l - \exp\big(V(l+1) - V(j+1)\big)\big(1 + W_j\big)\Big)\
	\le\ &\exp\big(2H_n\big) + \sum_{j=\nu_n}^{\nu_{n+1}-3} W_j^2\nonumber\\
	 \le\  &2 \textnormal{Var}_\omega\big(T_{\nu_{n+1}}- T_{\nu_n}\big). \label{BL2.11}
\end{align}
(To see the last inequality, use the same argument as in \eqref{ouf}). Therefore, \eqref{BL2.6} and \eqref{BL2.11} finally yield
\begin{align*}
	\Big(E_\omega^{\nu_n}(T_{\nu_{n+1}})\Big)^2 - \textnormal{Var}_\omega\big(T_{\nu_{n+1}}- T_{\nu_n}\big)\ \le 3 \textnormal{Var}_\omega\big(T_{\nu_{n+1}}- T_{\nu_n}\big).
\end{align*}
Note that the proof of the statement is the same for $\widetilde{\omega}$ instead of $\omega$ and for $\widetilde{T}_{\nu_{n+1}}^{(n)}$ instead of $T_{\nu_{n+1}}$. The considered quenched expectation in these cases is even smaller. In line \eqref{BL2.5}, the number of summands is different in these cases but since they are all negative we can use the same upper bound $0$ for all cases.
\end{proof}
\begin{lem}\label{PL9}
	For $\textbf{P}$-almost every environment $\omega$ we have
	\begin{align*}
 \limsup_{n\to\infty} \frac{E_{\widetilde{\omega}}(T_n) - E_{\widetilde{\omega}}\left(T_{n - \left\lceil 2(\ln n)^2\right\rceil}\right)}{\sqrt{ 	\textnormal{Var}_{\widetilde{\omega}}\left(T_n\right)}}\ &\le\ 2. 
\end{align*}
\end{lem}
\begin{proof}
First, we note that for environments $\omega\in E\left(n, \frac23\right)$ we have at most one increase of the potential of more than $\frac2{3\kappa}\ln n$ in the interval $[n - \left\lceil 2(\ln n)^2 \right\rceil,n]$. Therefore, we get for $\omega\in B(n)\cap E\left(n, \frac23\right) \cap C$ and $n$ large enough  
\begin{align*}
		\frac{E_{\widetilde{\omega}}^{n - \left\lceil 2(\ln n)^2 \right\rceil}(T_n)}{\sqrt{ 	\textnormal{Var}_{\widetilde{\omega}}\left(T_n\right)}}\
		 \le\ &\frac{1}{\sqrt{ 	\textnormal{Var}_{\widetilde{\omega}}\left(T_n\right)}}\sum_{k = n - \left\lceil 2(\ln n)^2 \right\rceil}^{n-1}E_{\widetilde{\omega}}^{k}T_{k+1}\left(\mathds{1}_{\left\{H_{k_0} \le \frac2{3\kappa} \ln n\right\}}+\mathds{1}_{\left\{H_{k_0} > \frac2{3\kappa} \ln n\right\}}\right) \\ 
		 \le\ & \frac{1}{\sqrt{ 	\textnormal{Var}_{\widetilde{\omega}}\left(T_n\right)}}\left(\lceil 2(\ln n)^2\rceil (\ln n)^2 \left(n^{\frac2{3\kappa}}+1\right) +  2\sqrt{\textnormal{Var}_{\widetilde{\omega}}\left(T_n\right)}\right),
	\end{align*}
	where we additionally used Lemma \ref{BL1} and \ref{BL2} to obtain the last line. Lemma \ref{PL5} and Theorem \ref{BT1} (note that Theorem \ref{BT1} is also true for $\widetilde{\omega}$) now finish the proof.
\end{proof}
\begin{lem}\label{BL2.1}
	Assume Assumptions 1 - 3. For any $\varepsilon <\frac13$, there exists an $\eta >0$ such that for  
	\begin{align*}
		A_n:=\ \left\{ \exists\ i,j\in\N,\  1\le i\le n :\ H_i\ >\ \frac{1-\varepsilon}{\kappa} \ln n,\  E_\omega^{\nu_{i}}\left(\widetilde{T}_{\nu_{i+1}}^{(n)}\right)^j\ >\ 3j! 2^j \left(E_\omega^{\nu_{i}}\widetilde{T}_{\nu_{i+1}}^{(n)}\right)^j\right\}
	\end{align*}
	we have
	\begin{align*}
		\textbf{P} \left( A_n\right)\ =\ o\left(n^{-\eta}\right).
	\end{align*}
\end{lem}
 For a proof see Lemma 5.9 in \cite{PZ} and Corollary 2.3.2 in \cite{Ko}.
\begin{lem}\label{BL3}
	Let $(a_n)_{n\in\N}$ be a (environment depending) subsequence of $n_k =2^{2^k}$ of blocks which fulfil the following lower bound for the height of block $(a_{n-1})_{n\in\N}$
	\begin{align}
		H_{a_{n-1}}\ >\ \frac{1-\varepsilon}\kappa \ln(\nu_{a_n})\label{BL3.0.1}
	\end{align}
	for $\varepsilon < \frac13$. Then, the sequence
	\begin{align*}
			\left(\left(\frac{\widetilde{T}_{\nu_{a_{n}}}^{(a_{n})} - \widetilde{T}_{{\nu_{a_{n}-1}}}^{(a_{n})}}{E_{\omega}^{\nu_{a_{n}-1}} \widetilde{T}_{\nu_{a_{n}}}^{(a_{n})}}\right)^2\right)_{n\in\N}
	\end{align*}
	is uniformly integrable with respect to $P_\omega$ for $\textbf{P}$-almost every environment $\omega$.
\end{lem}
\begin{proof}
At first, we note that the distribution of $\widetilde{T}_{\nu_{a_{n}}}^{(a_{n})} - \widetilde{T}_{{\nu_{a_{n}-1}}}^{(a_{n})}$ under $P_\omega$ is the same as the distribution of $\widetilde{T}_{\nu_{a_{n}}}^{(a_{n})}$ under $P_\omega^{\nu_{a_n-1}}$. We have for all blocks $a_n$ and all $c>0$
\begin{align}
	&\frac{1}{\left(E_\omega^{\nu_{a_n-1}}\widetilde{T}_{\nu_{a_n}}^{(a_n)}\right)^2} \int\limits_{\left\{\widetilde{T}_{\nu_{a_n}}^{(a_n)}\ >\ \sqrt{c}E_\omega^{\nu_{a_n-1}}\widetilde{T}_{\nu_{a_n}}^{(a_n)}\right\}}\left(\widetilde{T}_{\nu_{a_n}}^{(a_n)}\right)^2 dP_\omega^{\nu_{a_n-1}} \nonumber\\
	=\ & \frac{c \left(E_\omega^{\nu_{a_n-1}}\widetilde{T}_{\nu_{a_n}}^{(a_n)}\right)^2 P_\omega^{\nu_{a_n-1}}\left(\widetilde{T}_{\nu_{a_n}}^{(a_n)}\ >\ \sqrt{c}E_\omega^{\nu_{a_n-1}}\widetilde{T}_{\nu_{a_n}}^{(a_n)}\right)}{\left(E_\omega^{\nu_{a_n-1}}\widetilde{T}_{\nu_{a_n}}^{(a_n)}\right)^2} \nonumber\\
	 & + \frac{1}{\left(E_\omega^{\nu_{a_n-1}}\widetilde{T}_{\nu_{a_n}}^{(a_n)}\right)^2} \int\limits_{\sqrt{c}E_\omega^{\nu_{a_n-1}}\widetilde{T}_{\nu_{a_n}}^{(a_n)}}^{\infty}2t P_\omega^{\nu_{a_n-1}}\left(\widetilde{T}_{\nu_{a_n}}^{(a_n)}\ >\ t\right) dt \nonumber\\
	 =\ & c P_\omega^{\nu_{a_n-1}}\left(\widetilde{T}_{\nu_{a_n}}^{(a_n)}\ >\ \sqrt{c}E_\omega^{\nu_{a_n-1}}\widetilde{T}_{\nu_{a_n}}^{(a_n)}\right) + \int\limits_{\sqrt{c}}^{\infty}2 x P_\omega^{\nu_{a_n-1}}\left(\widetilde{T}_{\nu_{a_n}}^{(a_n)}\ >\ xE_\omega^{\nu_{a_n-1}}\widetilde{T}_{\nu_{a_n}}^{(a_n)}\right) dx. \label{BL3.1} \vphantom{\frac{1}{\left(E_\omega^{\nu_{a_n-1}}\widetilde{T}_{\nu_{a_n}}^{(a_n)}\right)^2}}
\end{align}
Due to \eqref{BL3.0.1}, Lemma \ref{BL2.1} and the fact that $(a_n)_{n\in\N}$ is growing at least as fast as $2^{2^n}$, the Borel-Cantelli lemma yields that for $\textbf{P}$-almost every environment $\omega$ there exists a $N=N(\omega)$ such that we have 
\begin{align*}
	E_\omega^{\nu_{a_{n}-1}}\left(\widetilde{T}_{\nu_{a_{n}}}^{(a_{n})}\right)^j\ \le\ 3j! 2^j \left(E_\omega^{\nu_{a_{n}-1}}\widetilde{T}_{\nu_{a_{n}}}^{(a_{n})}\right)^j
\end{align*}
for all $j\in \N$ and all $n\ge N$. We have $\widetilde{T}_{\nu_{a_n}}^{(a_n)} \ge 0$ by definition, and therefore we obtain for $n\ge N$ and $\textbf{P}$-almost every environment $\omega$ 
\begin{align*}
	E_\omega^{\nu_{a_{n}-1}} \exp\left(\frac{\widetilde{T}_{\nu_{a_{n}}}^{(a_{n})}}{4E_\omega^{\nu_{a_{n}-1}}\widetilde{T}_{\nu_{a_{n}}}^{(a_{n})}}\right)\ =\ &\sum_{j=0}^{\infty} \frac{1}{j!} \frac{E_\omega^{\nu_{a_{n}-1}}\left(\widetilde{T}_{\nu_{a_{n}}}^{(a_{n})}\right)^j}{4^j \left(E_\omega^{\nu_{a_{n}-1}}\widetilde{T}_{\nu_{a_{n}}}^{(a_{n})}\right)^j} \
	\le\  \sum_{j=0}^{\infty} \frac3{2^j}\ =\ 6.
\end{align*}
Then, applying Markov inequality yields the following upper bound for the probabilities in \eqref{BL3.1} for $n\ge N$ and $\textbf{P}$-almost every environment $\omega$
\begin{align*}
	P_\omega^{\nu_{a_{n}-1}}\left(\widetilde{T}_{\nu_{a_{n}}}^{(a_{n})}\ >\ xE_\omega\widetilde{T}_{\nu_{a_{n}}}^{(a_{n})}\right)\ \le\ 6\exp\left(-\frac{x}4\right), \hspace{25pt} x\ge 0.
\end{align*}
Thus, for all $n\ge N$ and $\textbf{P}$-almost every environment $\omega$ we get as an upper bound for \eqref{BL3.1} 
\begin{align*}
	&\frac{1}{\left(E_\omega^{\nu_{a_{n}-1}}\widetilde{T}_{\nu_{a_{n}}}^{(a_{n})}\right)^2} \int\limits_{\left\{\widetilde{T}_{\nu_{a_{n}}}^{(a_{n})}\ >\ \sqrt{c}E_\omega^{\nu_{a_{n}-1}}\widetilde{T}_{\nu_{a_{n}}}^{(a_{n})}\right\}}\left(\widetilde{T}_{\nu_{a_{n}}}^{(a_{n})}\right)^2 dP_\omega^{\nu_{a_{n}-1}} \\
	\le\  &6c\exp\left(-\frac{\sqrt{c}}{4}\right) + 12 \int_{\sqrt{c}}^{\infty}x \exp\left(-\frac{x}{4}\right)dx \
	=\ 6\exp\left(-\frac{\sqrt{c}}{4}\right) \left(c + 8\sqrt{c}+ 32\right).
\end{align*}
Therefore, we conclude that we have for $\textbf{P}$-almost every environment $\omega$
\begin{align*}
	&\lim_{c\to\infty} \sup_{n\in\N}  \int\limits_{\left\{\widetilde{T}_{\nu_{a_{n}}}^{(a_{n})}\ >\ \sqrt{c}E_\omega^{\nu_{a_{n}-1}}\widetilde{T}_{\nu_{a_{n}}}^{(a_{n})}\right\}}\left(\frac{\widetilde{T}_{\nu_{a_{n}}}^{(a_{n})}}{E_\omega^{\nu_{a_{n}-1}}\widetilde{T}_{\nu_{a_{n}}}^{(a_{n})}}\right)^2 dP_\omega^{\nu_{a_{n}-1}} \displaybreak[0]\\
	=\ & \max_{n\in\{1,...,N-1\}} \lim_{c\to\infty} \int\limits_{\left\{\widetilde{T}_{\nu_{a_{n}}}^{(a_{n})}\ >\ \sqrt{c}E_\omega^{\nu_{a_{n}-1}}\widetilde{T}_{\nu_{a_{n}}}^{(a_{n})}\right\}}\left(\frac{\widetilde{T}_{\nu_{a_{n}}}^{(a_{n})}}{E_\omega^{\nu_{a_{n}-1}}\widetilde{T}_{\nu_{a_{n}}}^{(a_{n})}}\right)^2 dP_\omega^{\nu_{a_{n}-1}} \\
	&+ \lim\limits_{c\to\infty} \sup_{n\ge N}  \int\limits_{\left\{\widetilde{T}_{\nu_{a_{n}}}^{(a_{n})}\ >\ \sqrt{c}E_\omega^{\nu_{a_{n}-1}}\widetilde{T}_{\nu_{a_{n}}}^{(a_{n})}\right\}}\left(\frac{\widetilde{T}_{\nu_{a_{n}}}^{(a_{n})}}{E_\omega^{\nu_{a_{n}-1}}\widetilde{T}_{\nu_{a_{n}}}^{(a_{n})}}\right)^2 dP_\omega^{\nu_{a_{n}-1}} \
	=\ 0 .
\end{align*}
\end{proof}
\section{Cutoff and mixing times}\label{cutoff}
In this section, we show that a sequence of transient lazy RWRE on $(\{0,...,n\})_{n\in\N}$ exhibits a cutoff under Assumptions 1 and 2 for $\kappa>1$, but there is no cutoff for $\kappa<1$, if we additionally assume Assumption 3. Further, we prove that the mixing time is roughly of order $n^{\frac1\kappa}$ for $\kappa <1$ and roughly of order n for $\kappa\ge 1$.
\vspace{11pt}\\
At first, one easily checks that the reversible (and hence stationary) probability distribution of the RWRE $(X_n)_{n\in\N_0}$ under $P_{\omega^n}$ is given by 
\begin{align*}
	\pi_{\omega^n}(0)&:=\ \frac{\exp(-V(1))}{C_n}\ , \\
	\pi_{\omega^n}(x)&:=\ \frac{\exp(-V(x+1)) + \exp(-V(x))}{C_n}\hspace{20pt}\text{ for } x=1,...,n-1,\\
		\pi_{\omega^n}(n)&:=\ \frac{\exp(-V(n))}{C_n}\ ,
\end{align*}
where $C_n:=\  2\sum_{x=1}^n \exp\big(-V(x)\big)$ is the normalizing constant. \vspace{11pt}\\
To show a cutoff or to prove that no cutoff is possible, we need to understand very precisely which events have a high probability with respect to the stationary  distribution. In this section, we show that the mass of the stationary  distribution of the RWRE under $P_{\omega^n}$ is asymptotically concentrated on the interval $[n - 2(\ln n)^2,n]$.
\begin{lem}\label{CL1}
	We have\ \ $\lim_{n\to\infty}\pi_{\omega^n}\left(\left[n - 2(\ln n)^2,n \right]\right) = 1$ for $\textbf{P}$-almost every environment $\omega$.
\end{lem}
\begin{proof}
	We get
\begin{align}
	\pi_{\omega^n}\left([0,n - 2(\ln n)^2 )\right) \vphantom{\frac{\sum\limits}{\sum\limits_i}}\
	 =\ &\frac{\sum\limits_{x=1}^{\left\lfloor n - 2(\ln n)^2\right\rfloor}\left(\exp\big(-V(x+1)\big) + \exp\big(-V(x)\big)\right)+\exp\big(-V(1)\big)}{2\sum\limits_{x=1}^{n}\exp\big(-V(x)\big)} \nonumber\displaybreak[0]\\
	\le\ &\frac{2n\exp\left(-\min\limits_{0\le i\le\left\lfloor n - 2(\ln n)^2\right\rfloor}V(i)\right)+\exp(-V(1))}{2\exp\big(-V(n)\big)} \nonumber\\
	\leq\ &n\exp\left(V(n) -\min_{0\le i\le\left\lfloor n - 2(\ln n)^2\right\rfloor}V(i)\right) + \exp\big(V(n)-V(1)\big).  \vphantom{\frac{\sum\limits_i^k}{\sum\limits_i^k}}\label{CL1.1}
\end{align}
By the definition of $B_1(n)$, we know that for all $i \le n - 2(\ln n)^2$ we have
$
	V(n) - V(i)\ <\ -2\ln n. 
$
For environments $\omega \in B_1(n)$, this therefore yields as an upper bound for \eqref{CL1.1} 
\begin{align*}
	\pi_{\omega^n}\Big([0,n - 2(\ln n)^2 )\Big)\ &\le\ 2n\exp\big(-2\ln n\big)+\exp\big(V(n)-V(1)\big)\
	 =\ \frac2n+\exp\big(V(n)-V(1)\big)
\end{align*}
which shows that for $\omega \in B_1(n)$ we have 
\begin{align*}
	\lim_{n\to\infty}\pi_{\omega^n}\Big(\left[n - 2(\ln n)^2,n \right]\Big)\ =\ 1.
\end{align*}
Lemma \ref{PL10} and the Borel-Cantelli lemma now finish the proof.
\end{proof} 
Note here that with respect to $P_{\omega^n}$ the lazy random walk $(Y_k)_{k\in\N_0}$ has the same stationary distribution $\pi_{\omega^n}$ as $(X_k)_{k\in\N_0}$.
\vspace{11pt}\\
Next, we show that we can bound the distance in total variation of the lazy RWRE to its stationary distribution by using hitting times.
\begin{lem}\label{CL2}
We have for all $n,k\in\N$ 
	 \begin{align*}
	 	 \max_{x\in\{0,...,n\}} \big\Vert P_{\omega^n}^x(Y_k\in \cdot) - \pi_{\omega^n}\big\Vert_{TV}\ \le\ &P_{\omega^n}\left(T_n^{Y}  > k\right).
	 \end{align*}
\end{lem}
\begin{proof}
	We have for all $k\in\N$ using Corollary 5.3 in \cite{LPW}:
\begin{align}
 \max_{x\in\{0,...,n\}} \big\Vert P_{\omega^n}^x(Y_k\in \cdot) - \pi_{\omega^n}\big\Vert_{TV} \ 
	\le\  &\max_{x,y\in\{0,...,n\}} P_{\omega^n}^{\vec{z}} \Big(\min\{s\in\N : Y_s^{x}=Y_s^{y}\}\ >\ k\Big), \label{CL2.1}
\end{align}
where for all $x,y\in\{0,...,n\}$ under $P_{\omega^n}^{\vec{z}}$ we consider a coupling  $\left(Y_k^{x},Y_k^{y}\right)_{k\in\N_0}$ of two lazy RWRE on $\{0,...,n\}$ with $P_{\omega^n}^{\vec{z}}\left(Y_0^{x}=x, Y_0^{y}=y\right) = 1$ and marginal distribution $P_{\omega^n}^x$ and $P_{\omega^n}^y$, respectively, defined in the following way: until the two chains meet for the first time, the chains move according to the following two steps: first, we toss a coin to decide which chain moves. After that, the chosen chain performs a step of a RWRE in the environment ${\omega^n}$ described in \eqref{RWRE} and the other chain stays at its position. After they met for the first time, we move them together according to the law of a lazy RWRE. Further, note that due to this coupling the two chains cannot cross each other without meeting, and therefore we have
\begin{align*}
	\max_{x,y\in\{0,...,n\}} P_{\omega^n}^{\vec{z}} \left(\min\{s\in\N : X_s^{x}=X_s^{y}\}\ >\ t\right)\ \le\  P_{\omega^n}\left(T_n^{Y}  > t\right),
\end{align*}
which together with \eqref{CL2.1} yields the statement.
\end{proof}
\begin{proof}[Proof of Theorem \ref{CT2}]
First, we note that due to \eqref{B1} and Theorem \ref{BT1} we have for $\kappa >1$ and $\textbf{P}$-almost every environment $\omega$
\begin{align}
	\sqrt{\textnormal{Var}_{\omega^n}(T_n)}\ \le\ \sqrt{\textnormal{Var}_{\omega}(T_n)}\ =\ o(n). \label{CT2.0.1}
\end{align}
Therefore (1) in Definition \ref{CD2} is valid for $$t_\omega(n):= 2E_{\omega^n}(T_n) \hspace{15pt} \text{ and } \hspace{15pt}   f_\omega(n):=\sqrt{\textnormal{Var}_{\omega^n}(T_n)}$$
because we have for $\textbf{P}$-almost every environment $\omega$ due to \eqref{Ergodensatz} $$1\ \le\ \lim_{n\to\infty} \frac{E_{\omega^n}(T_n)}n\ \le\ \lim_{n\to\infty} \frac{E_{\omega}(T_n)}n\ =\ \mathbb{E} T_1.$$ 
Further, we define $$t_\omega^+(c,n):=\ t_\omega(n) + c\cdot f_\omega(n).$$ Then, using Lemma \ref{Expectation lRWRE}, Lemma \ref{CL2} and Chebyshev's inequality, we get 
\begin{align}
	 d_n\left(t_\omega^+(c,n)\right)\ \le\ P_{\omega^n}\left(T_n^{Y}  > t_\omega^+(c,n)\right)\vphantom{\frac12} \displaybreak[0]\
	\le\ & P_{\omega^n}\left(\Big|T_n^{Y}  - E_{\omega^n} T_n^Y \Big|> c \sqrt{\textnormal{Var}_{\omega^n}(T_n)}\right)\vphantom{\frac12}\\
	\le\ & \frac1{c^2} \frac{\textnormal{Var}_{\omega^n}(T_n^Y)}{\textnormal{Var}_{\omega^n}(T_n)}.\label{CT2.3.1}
\end{align}
Therefore, Theorem \ref{BT1} together with \eqref{Expectation lRWRE2} and \eqref{BT1.8} yields for $\textbf{P}$-almost every environment $\omega$
\begin{align*}
	\lim_{c\to\infty}\limsup_{n\to\infty} d_n\left(t_\omega^+(c,n)\right)\ \le\ \lim_{c\to\infty} \frac{4+O(1)}{c^2} \ =\ 0,
\end{align*}
and thus we showed (3) in Definition \ref{CD2}.\vspace{11pt}\\
As the last step, we are interested in the lower bound of the cutoff (cf.~(2) in Definition \ref{CD2}). The idea is to show that before the cutoff window the lazy RWRE with start in $0$ has with high probability not reached position $\lceil n-2(\ln n)^2\rceil$ and therefore is still in the interval $[0,n - 2(\ln n)^2)$. On the other hand, the mass of the stationary distribution $\pi_{\omega^n}$ is for large $n$ concentrated on the interval $[n-2(\ln n)^2,n]$ due to Lemma \ref{CL1}. We define 
\begin{align*}
	t^-_\omega(c,n) :=\ t_\omega(n) - c\cdot f_\omega(n)\ \text{ and } \ a_n:=\ n - \left\lceil 2(\ln n)^2\right\rceil. 
\end{align*} 
Then, we get for $c$ and $n$ large enough using Lemma \ref{Expectation lRWRE}
\begin{align}
	P_{\omega^n} \left(Y_{t_\omega^-(c,n)}\ \ge\ a_n\right)\
	\le\ & P_{\omega^n}\left(T^{Y}_{a_n} \le t_\omega^-(c,n)\right) \vphantom{\sum_i^k}\
	\le\   P_{\omega^n}\left( \left|T_{a_n}^Y- E_{\omega^n}T_{a_n}^Y\right|\ge  \frac{c}2\sqrt{\textnormal{Var}_{\omega^n}(T_n)} \right) \nonumber,
\end{align}
where to obtain the last inequality we used that we have for $\textbf{P}$-almost every environment $\omega$ due to Lemma \ref{PL9}  
$$
	c -\frac{2E_{\omega^n}^{a_n}T_n}{\sqrt{\textnormal{Var}_{\omega^n}(T_n)}} >\frac{c}2
$$
for all $n \ge n(\omega)$ and $c$ large enough. Now, again applying Chebyshev's inequality and Lemma \ref{Expectation lRWRE}, we get for c large enough 
\begin{align*}
P_{\omega^n} \left(Y_{t_\omega^-(c,n)}\ \ge\ a_n\right) \
\le\ & \frac{\textnormal{Var}_{\omega^n}(T_{a_n}^Y)}{\textnormal{Var}_{\omega^n}(T_n)}\cdot \frac4{c^2} \
=\  \frac{4\textnormal{Var}_{\omega^n}(T_{a_n})+ 2 E_{\omega^n} T_{a_n}}{\textnormal{Var}_{\omega^n}(T_n)}\cdot \frac4{c^2}.
\end{align*}
Finally, using Lemma \ref{CL1} and Theorem \ref{BT1}, we conclude that we have for $\textbf{P}$-almost every environment 
\begin{align*}
	\lim_{c\to\infty}\liminf_{n\to\infty} d_n\left(t^-_\omega(c,n)\right) \vphantom{\sum^k}\
	=\ & \lim_{c\to\infty}\liminf_{n\to\infty} \max_{x\in\{0,...,n\}}||P^x_{\omega^n}\left(Y^n_{t_\omega^-(c,n)} \in\cdot\right) - \pi_{\omega^n}||_{TV}\nonumber\vphantom{\sum_i^k} \displaybreak[0]\\
	\ge\ &\lim_{c\to\infty}\liminf_{n\to\infty}\left( \pi_{\omega^n} \left( [n - 2(\ln n)^2,n]\right) - P_{\omega^n}\left(Y_{t_\omega^-(c,n)}\ \ge\ n - 2(\ln n)^2\right)\right)\vphantom{\sum_i^k}\\
	\ge\ &1 - \lim_{c\to\infty}\limsup_{n\to\infty} \frac{4\textnormal{Var}_{\omega^n}(T_{a_n})+ 2 E_{\omega^n} T_{a_n} }{\textnormal{Var}_{\omega^n}(T_n)}\cdot \frac4{c^2}\nonumber\\
	  =\ &1 - \lim_{c\to\infty} \frac{16+O(1)}{c^2}\ =\ 1, \nonumber	\vphantom{\sum_i^k}
\end{align*}
which shows (2) in Definition \ref{CD2}.
\end{proof}
Next, we consider the case $\kappa <1$. We show that in this case the transition to stationarity is not as sharp as required for the cutoff phenomenon. For $\kappa<1$, there exists an environment depending sequence of deep blocks in which the RWRE with start in $0$ spends most of its time before it hits the endpoint of this high block (cf. Lemma \ref{L4}). Afterwards, we show that this sequence excludes that the lazy RWRE exhibits cutoff. For $k\in\N$ we define 
\begin{align}
	n_k:=\ 2^{2^k}\ \text{ and }\ d_k:=\ n_k - n_{k-1}\ =\ 3 n_{k-1}. \label{d_k}
\end{align}
%
%
\begin{lem} \label{L4}
	Let Assumptions 1-3 hold and assume $\kappa<1$. For $\textbf{P}$-almost every environment $\omega$ there exists a random subsequence $(a_m)_{m=1,2,\ldots } = a_m(\omega)_{m=1,2,\ldots }$ of $(n_k)_{k=1,2,\ldots }$ such that
	\begin{align*}
		\exp(H_{a_m-1})\ \ge\ m^2 E_\omega \widetilde{T}_{\nu_{a_m-1}}^{(a_m)}.
	\end{align*}
\end{lem}
For a proof see Corollary 4.4 in \cite{PZ}. 
\begin{proof}[Proof of Theorem \ref{CT3}]
In the following, we use the sequence of high blocks of Lemma \ref{L4} in order to construct two sequences of the same order as the mixing time with the property that along the first (smaller) one the distance to stationarity is bounded away from 1 and along the second (bigger) one bounded away from 0. 

Lemma \ref{L4} gives us for $\textbf{P}$-almost every environment $\omega$ a sequence $(j_m)_{m\in\N} = (j_m(\omega))_{m\in\N}$ which fulfils
\begin{align}
	\exp\left(H_{j_m-1}\right)\ \ge\ m^2 E_\omega\left(\widetilde{T}_{\nu_{j_m-1}}^{(j_m)}\right) \label{CT3.0.1}
\end{align}
for all $m\in\N$. Obviously, this can only be the case, if block $j_m-1$ is the highest block in the interval $[0,\nu_{j_m}]$. Therefore, for environments $\omega$ which are additionally in $F(\nu_{j_m})$ (cf.~\eqref{F} for the definition of $F(\nu_{j_m})$), we have that condition \eqref{BL2.0.1} holds. Hence, we can use Lemma \ref{BL2} to obtain that for $\omega \in F(\nu_{j_m})$ we have
\begin{align}
	\left(E_\omega^{\nu_{j_m-1}}\left(\widetilde{T}_{\nu_{j_m}}^{(j_m)}\right)\right)^2\ \le\ 4\textnormal{Var}_\omega\left(\widetilde{T}_{\nu_{j_m}}^{(j_m)}- \widetilde{T}_{\nu_{j_m-1}}^{(j_m)}\right). \label{CT3.0.2}
\end{align}
Further, we note that the laws of
\begin{align*}
\left(\frac{\widetilde{T}_{\nu_{j_m}}^{(j_m)}-\widetilde{T}_{\nu_{j_m-1}}^{(j_m)}}{E_{\omega}^{\nu_{j_m-1}}\widetilde{T}_{\nu_{j_m}}^{(j_m)}}\right)_{m\in\N}\ =\ \left(\frac{\widetilde{T}_{\nu_{j_m}}^{(j_m)}-\widetilde{T}_{\nu_{j_m-1}}^{(j_m)}}{E_{\omega}\left(\widetilde{T}_{\nu_{j_m}}^{(j_m)}-\widetilde{T}_{\nu_{j_m-1}}^{(j_m)}\right)}\right)_{m\in\N}
\end{align*}
 are tight with respect to $P_\omega$ due to the Markov inequality. Thus, Prohorov's Theorem yields that there exists a subsequence $\left(\nu_{j_{m_k}}\right)_{k\in\N}$ such that 
\begin{align}	\frac{\widetilde{T}_{\nu_{j_{m_k}}}^{(j_{m_k})}-\widetilde{T}_{\nu_{j_{m_k}-1}}^{(j_{m_k})}}{E_{\omega}^{\nu_{j_{m_k}-1}}\widetilde{T}_{\nu_{j_{m_k}}}^{(j_{m_k})}}\ \stackrel{d}{\longrightarrow}\ Z \label{CT3.1}
\end{align}
for some positive random variable Z as $k\to\infty$, where $\stackrel{d}{\rightarrow}$ means convergence in distribution. Note that $(j_{m_k})_{k\in\N}$ is a subsequence of $(n_k)_{k\in\N} = (2^{2^k})_{k\in\N}$ by construction. Therefore and due to \eqref{CT3.0.1}, assumption \eqref{BL3.0.1} of Lemma \ref{BL3} is valid for $\omega \in F(\nu_{j_{m_k}})$ (using the property of set $B_4(\nu_{j_{m_k}})$), and we get that for $\textbf{P}$-almost every environment $\omega$ the sequence 
\begin{align*}	\left(\left(\frac{\widetilde{T}_{\nu_{j_{m_{k}}}}^{(j_{m_{k}})}-\widetilde{T}_{\nu_{j_{m_{k}}-1}}^{(j_{m_{k}})}}{E_{\omega}^{\nu_{j_{m_{k}}-1}}\widetilde{T}_{\nu_{j_{m_{k}}}}^{(j_{m_{k}})}}\right)^2\right)_{k\in\N}
\end{align*}
is uniformly integrable with respect to $P_\omega$. Thus, the first two moments of $$\frac{\widetilde{T}_{\nu_{j_{m_k}}}^{(j_{m_k})}-\widetilde{T}_{\nu_{j_{m_k}-1}}^{(j_{m_k})}}{E_{\omega}^{\nu_{j_{m_k}-1}}\widetilde{T}_{\nu_{j_{m_k}}}^{(j_{m_k})}}$$ 
converge to the corresponding moments of $Z$ for $\textbf{P}$-almost every environment $\omega$ as $k\to\infty$. \\[11pt]
Further, for $\textbf{P}$-almost every environment $\omega$ we have $\omega\in F(n)$ for all $n$ large enough due to Lemma \ref{PL5}. Hence, using \eqref{CT3.0.2}, we get for $\textbf{P}$-almost every environment $\omega$
\begin{align*}
	\textnormal{Var}_\omega(Z)\ =\ &\lim_{k\to\infty}\frac1{\left(E_{\omega}^{\nu_{j_{m_k}-1}}\widetilde{T}_{\nu_{j_{m_k}}}^{(j_{m_k})}\right)^2}\textnormal{Var}_\omega\left(\widetilde{T}_{\nu_{j_{m_k}}}^{(j_{m_k})}-\widetilde{T}_{\nu_{j_{m_k}-1}}^{(j_{m_k})}\right) \ \ge\  \frac14.
\end{align*}
Therefore, for $\textbf{P}$-almost every environment $\omega$ we have that the distribution of Z with respect to $P_\omega$ cannot be the Dirac measure in 1. Since $E_\omega Z = 1$, there exists an interval $(a,b)$ with $a<1<b$, such that 
\begin{align}
	&\lim_{k\to\infty} P_{\omega}\left(\widetilde{T}_{\nu_{j_{m_k}}}^{(j_{m_k})}-\widetilde{T}_{\nu_{j_{m_k}-1}}^{(j_{m_k})}\ >\ a\cdot E_{\omega}^{\nu_{j_{m_k}-1}}\widetilde{T}_{\nu_{j_{m_k}}}^{(j_{m_k})}\right)\ <\ 1, \label{CT3.2} \\
	&\lim_{k\to\infty} P_{\omega}\left(\widetilde{T}_{\nu_{j_{m_k}}}^{(j_{m_k})}-\widetilde{T}_{\nu_{j_{m_k}-1}}^{(j_{m_k})}\ <\ b\cdot E_{\omega}^{\nu_{j_{m_k}-1}}\widetilde{T}_{\nu_{j_{m_k}}}^{(j_{m_k})}\right)\ <\ 1
\label{CT3.3}.	 
\end{align}
Using Lemma \ref{CL2}, we get (cf.~\eqref{A(n)} for the definition of $A(n)$ and cf.~\eqref{Zdef} for the definition of $(Z_k)_{k\in\N}$)
\begin{align*}
	 d_n\left(\left\lceil (a+b)\cdot E_\omega\widetilde{T}_n^{(n)}\right\rceil\right)\vphantom{\sum^k}\
	\le\ & P_{\omega^n}\left(T_n^{Y}> (a+b)\cdot E_\omega\widetilde{T}_n^{(n)}\right)\vphantom{\sum_i^k}\\
	 \le\ & P_{\widetilde{\omega}}\left(A(n)^{\textnormal{c}}\right) + P_{\omega^n}\left(\widetilde{T}_n^{(n)}\ >\ \left(a + \frac14(b-a)\right)\cdot 	
						E_\omega\widetilde{T}_n^{(n)}\right) \\ 
						& + P_{\widetilde{\omega}}\left(\sum_{i=1}^{\left\lceil (a+b)\cdot E_\omega\widetilde{T}_n^{(n)}\right\rceil}\big( 1 - Z_i\big) > \left(\frac34 b+\frac14 a\right)\cdot E_\omega\widetilde{T}_n^{(n)}\right).
\end{align*}
We note that for arbitrary $a>0$ we have $P_{\omega^n}\left(\widetilde{T}_n^{(n)} > a\right) \le P_\omega\left(\widetilde{T}_n^{(n)} > a\right)$, and therefore, we get
\begin{align*}						
\liminf_{n\to\infty} d_n\left(\left\lceil (a+b) E_\omega\widetilde{T}_n^{(n)}\right\rceil\right)
 \le \liminf_{k\to\infty}\left(P_{\widetilde{\omega}}\left(A\left(\nu_{j_{m_{k}}}\right)^{\textnormal{c}}\right)\hspace{-1pt}+\hspace{-1pt}P_\omega\left(\widetilde{T}_{\nu_{j_{m_k}}}^{(j_{m_k})} \hspace{-1pt}>\hspace{-1pt} \left(a+\frac14(b-a)\right)E_\omega\widetilde{T}_{\nu_{j_{m_k}}}^{(j_{m_k})}\right)\right) ,
\end{align*}
where we additionally used that $\frac34 b+\frac14 a > \frac{a+b}2$ and thus by Cramér's Theorem
$$P_{\widetilde{\omega}}\left(\sum_{i=1}^{\left\lceil (a+b)\cdot E_\omega\widetilde{T}_n^{(n)}\right\rceil}\big( 1- Z_i\big) > \left(\frac34 b+\frac14 a\right)\cdot E_\omega\widetilde{T}_n^{(n)}\right)\ \stackrel{n\to \infty}{\longrightarrow}\   0.$$	 
Further, we note that due to \eqref{CT3.0.1} we have 
\begin{align*}
	\frac{E_\omega\widetilde{T}_{\nu_{j_{m_k}-1}}^{(j_{m_k})}}{E_\omega\widetilde{T}_{\nu_{j_{m_k}}}^{(j_{m_k})}}\ \le\ \frac{E_\omega\widetilde{T}_{\nu_{j_{m_k}-1}}^{(j_{m_k})}}{W_{j_{m_k}-1}}\ \le \frac{E_\omega\widetilde{T}_{\nu_{j_{m_k}-1}}^{(j_{m_k})}}{\exp(H_{j_{m_k}-1})}\ \le\ \frac1{m_k^2},
\end{align*}
which together with Lemma \ref{PL1} and equation \eqref{CT3.2} yields
\begin{align} 
&\liminf_{n\to\infty} d_n\left((a+b)\cdot E_\omega\widetilde{T}_n^{(n)}\right)\vphantom{\sum^k}\nonumber\\
	 \le\hspace{2.5pt} & \liminf_{k\to\infty}\left(  P_\omega\left(\widetilde{T}_{\nu_{j_{m_k}-1}}^{(j_{m_k})} > \frac14(b-a)E_\omega\widetilde{T}_{\nu_{j_{m_k}}}^{(j_{m_k})}\right)    +P_\omega\left(\left(\widetilde{T}_{\nu_{j_{m_k}}}^{(j_{m_k})}-\widetilde{T}_{\nu_{j_{m_k}-1}}^{(j_{m_k})}\right)>a\cdot E_\omega\widetilde{T}_{\nu_{j_{m_k}}}^{(j_{m_k})}\right)   \right)\vphantom{\sum_i^k}\nonumber \displaybreak[0]\\
\le\hspace{2.5pt} & \liminf_{k\to\infty}\left( \frac{4}{b-a}\frac{E_\omega\widetilde{T}_{\nu_{j_{m_k}-1}}^{(j_{m_k})}}{E_\omega\widetilde{T}_{\nu_{j_{m_k}}}^{(j_{m_k})}}  + P_\omega\left(\left(\widetilde{T}_{\nu_{j_{m_k}}}^{(j_{m_k})}-\widetilde{T}_{\nu_{j_{m_k}-1}}^{(j_{m_k})}\right)>a\cdot E_\omega\widetilde{T}_{\nu_{j_{m_k}}}^{(j_{m_k})}\right)   \right) \nonumber\\	 
\le\hspace{2.5pt} & \liminf_{k\to\infty}\left( \frac{4}{b-a}\cdot\frac1{m_k}  + P_\omega\left(\left(\widetilde{T}_{\nu_{j_{m_k}}}^{(j_{m_k^2})}-\widetilde{T}_{\nu_{j_{m_k}-1}}^{(j_{m_k})}\right)>a\cdot E_\omega\widetilde{T}_{\nu_{j_{m_k}}}^{(j_{m_k})}\right)   \right) \
						 <\hspace{2.5pt}  1. \vphantom{\sum^k}\label{CT3.4}
\end{align} 
Next, we want to construct a sequence of (larger) time points of the same order as $(a+b)\cdot E_\omega\widetilde{T}_n^{(n)}$ at which the distance to stationarity is strictly positive. We define the sequence $(a_m)_{m\in\N}$ by
$
	a_m:=\ \lfloor\nu_{j_m} + 2(\ln \nu_{j_m})^2\rfloor,
$
and using \eqref{CT3.0.1} we get for $\omega \in F(a_m)$
\begin{align*}
	\exp(H_{j_m-1})\ \ge \ \frac12m^2\left(E_\omega \widetilde{T}_{\nu_{j_m-1}}^{(j_m)}  + E_\omega^{\nu_{j_m}} \widetilde{T}_{a_m}\right)
\end{align*}
because after the ``large" increase in block $j_{m}-1$ there cannot be a second ``large" increase in the interval $[\nu_{j_m},a_m]$ for environments $\omega\in E\left(a_m, \frac23\right) \subset F(a_m)$.	
Therefore, we can conclude for $\omega \in F(a_m)$
\begin{align*}
		\frac{E_\omega\widetilde{T}_{\nu_{j_m-1}}^{(j_m)}+E_\omega^{\nu_{j_m}} \widetilde{T}_{a_m}}{E_\omega\widetilde{T}_{a_m}^{(j_m)}}\ \le\ \frac{E_\omega\widetilde{T}_{\nu_{j_m-1}}^{(j_m)}+ E_\omega^{\nu_{j_m}} \widetilde{T}_{a_m}}{\exp(H_{j_m-1})}\ \le\ \frac2{m^2}\ \stackrel{m\to\infty}{\longrightarrow}\ 0,
\end{align*}
which yields
\begin{align}
\frac{E_\omega^{\nu_{j_m-1}}\widetilde{T}_{\nu_{j_m}}^{(j_m)}}{E_\omega\widetilde{T}_{a_m}^{(j_m)}}\ \stackrel{m\to\infty}{\longrightarrow}\ 1.\label{CT3.4.1}
\end{align}
Let $\varepsilon>0$ be small enough such that $\left(\frac32 - \varepsilon \right)b + \frac12 a > b+a$. We get
\begin{align}
	& P_{\omega^n}\left(Y_{\left\lceil\left(\left(\frac32 - \varepsilon \right) b+\frac12 a\right)\cdot E_\omega\widetilde{T}_n^{(n)}\right\rceil}\ \ge\ n-2(\ln n)^2\right)\nonumber\\
	\le\ & \left(P_{\omega^n}\left(T_{\left\lfloor n - 2 (\ln n)^2\right\rfloor}\ <\  \left(1 - \varepsilon \right)b\cdot E_\omega\widetilde{T}_n^{(n)}\right) \vphantom{\sum_{i=1}^{\left\lceil\left(\left(\frac32 - \varepsilon \right) b+\frac12 a\right)\cdot E_\omega\widetilde{T}_n^{(n)}\right\rceil}} \right. \nonumber\\
	&\hspace{15pt} + \left. P_{\widetilde{\omega}}\left(\sum_{i=1}^{\left\lceil\left(\left(\frac32 - \varepsilon \right) b+\frac12 a\right)\cdot E_\omega\widetilde{T}_n^{(n)}\right\rceil}(1-Z_i) < \left(\frac12 b+ \frac12 a\right)\cdot E_\omega\widetilde{T}_n^{(n)}\right)\right)\label{CT3.4.1.0.1}
\end{align} 
Due to the choice of $\varepsilon$, Cramér's Theorem yields
\begin{align}
	P_{\widetilde{\omega}}\left(\sum_{i=1}^{\left\lceil\left(\left(\frac32 - \varepsilon \right) b+\frac12 a\right)\cdot E_\omega\widetilde{T}_n^{(n)}\right\rceil}Z_i < \left(\frac12 b+ \frac12 a\right)\cdot E_\omega\widetilde{T}_n^{(n)}\right)\ \stackrel{n\to\infty}{\longrightarrow}\ 0. \label{CT3.4.1.0.2}
\end{align}
Again due to Lemma \ref{PL5}, we have that $\omega \in F(n)$ for $\textbf{P}$-almost every environment $\omega$ and all $n$ large enough. Thus, we get for $\textbf{P}$-almost every environment $\omega$ using \eqref{CT3.4.1.0.1} and \eqref{CT3.4.1.0.2} (cf.~\eqref{A(n)} for the definition of the set $A(n)$)
\begin{align*}
	&\liminf_{n\to\infty}P_{\omega^n}\left(Y_{\left\lceil\left(\left(\frac32 - \varepsilon \right) b+\frac12 a\right)\cdot E_\omega\widetilde{T}_n^{(n)}\right\rceil}\ \ge\ n-2(\ln n)^2\right)\\
	\le\ & \liminf_{k\to\infty} \left(P_{\omega^n}\left(\widetilde{T}_{\nu_{j_{m_k}}}^{(j_{m_k})}\ <\ (1 - \varepsilon)b\cdot E_\omega\widetilde{T}_{a_{m_{k}}}^{(j_{m_k})}\right)+ P_{\widetilde{\omega}}\left(A\left(\nu_{j{_{m_{k}}}}\right)^{\textnormal{c}}\right)\right) \vphantom{\sum_i^k}.\nonumber
\end{align*}
Further, using that $\left(\widetilde{T}_{\nu_{j_{m_k}}}^{(j_{m_k})} -\widetilde{T}_{\nu_{j_{m_k}-1}}^{(j_{m_k})} \right)$ has the same distribution with respect to $P_\omega$ as with respect to $P_{\omega^n}$, we get together with \eqref{CT3.4.1} and Lemma \ref{PL1} 
\begin{align*}
&\liminf_{n\to\infty}P_{\omega^n}\left(Y_{\left\lceil\left(\left(\frac32 - \varepsilon \right) b+\frac12 a\right)\cdot E_\omega\widetilde{T}_n^{(n)}\right\rceil}\ \ge\ n-2(\ln n)^2\right)\\ \le\ & \liminf_{k\to\infty} \left(P_{\omega}\left(\left(\widetilde{T}_{\nu_{j_{m_k}}}^{(j_{m_k})} -\widetilde{T}_{\nu_{j_{m_k}-1}}^{(j_{m_k})} \right)\ <\ b\cdot E_\omega^{\nu_{j_{m_k}-1}}\widetilde{T}_{\nu_{j_{m_k}}}^{(j_{m_k})}\right)\right). \vphantom{\sum_i^k}
\end{align*}
Finally, this together with \eqref{CT3.3} and Lemma \ref{CL1} yields
\begin{align}
 &\limsup_{n\to\infty} d_n\Bigg(\left\lceil\left(\left(\frac32 - \varepsilon \right) b+\frac12 a\right)\cdot E_\omega\widetilde{T}_n^{(n)}\right\rceil\Bigg)\nonumber\vphantom{\sum^k}\\
 \ge\ & \limsup_{n\to\infty}\left(\pi_{\omega^n}\Big(\left[n-2(\ln n)^2,n\right]\Big)- P_{\omega^n}\left(Y_{\left\lceil\left(\left(\frac32 - \varepsilon \right) b+\frac12 a\right)\cdot E_\omega\widetilde{T}_n^{(n)}\right\rceil}\ \ge\ n-2(\ln n)^2\right)\right)\nonumber\vphantom{\sum^k}\displaybreak[0]\\
   \ge\ &1 - \lim_{k\to\infty} P_{\omega}\left(\left(\widetilde{T}_{\nu_{j_{m_k}}}^{(j_{m_k})} -\widetilde{T}_{\nu_{j_{m_k}-1}}^{(j_{m_k})} \right)\ <\ b\cdot E_\omega^{\nu_{j_{m_k}-1}}\widetilde{T}_{\nu_{j_{m_k}}}^{(j_{m_k})}\right) \vphantom{\sum_i^k}\ 
  >\ 0 \vphantom{\sum^k}.\label{CT3.5}
\end{align}
Since for $a<b$ we have $a+b < \left(\frac32-\varepsilon\right) b + \frac 12 a,$ \eqref{CT3.4} and \eqref{CT3.5} show that the order of the window size has to be bigger or equal to the order of $E_\omega\widetilde{T}_n^{(n)}$. But due to Theorem \ref{mixing time}, this is the order of the mixing time. Consequently, the sequence of lazy RWRE $(Y_n)_{n\in\N_0}$ cannot exhibit a cutoff defined in Definition \ref{CD2}.
\end{proof}
\begin{proof}[Proof of Theorem \ref{mixing time}] 
At first, let us assume $\kappa\le 1$. For large $n$, we have due to Lemma \ref{CL2} and using Markov inequality
\begin{align*} 
	d_n\left(\lceil 12E_{\omega^n} (T_n)\rceil\right)\ \le\ &P_{\omega^n}\left(T_n^{Y}  > \lceil 12E_{\omega^n} (T_n)\rceil\right) \\
	 \le\ &P_{\omega^n}\left(T_n > 5E_{\omega^n} (T_n) \right) + P_{\omega^n}\left(\sum_{i=1}^{\lceil 12E_{\omega^n} (T_n)\rceil}Z_i > 7E_{\omega^n} (T_n)\right)\	 \le\ \frac14
\end{align*}
and therefore $t_{\textnormal{mix}}^\omega(n)\ \le\ \lceil 12E_{\omega^n} (T_n)\rceil\ \le\ \lceil 12E_{\omega} T_n\rceil. $ Due to Theorem \ref{BT1} (\textit{a}), this yields for $\kappa \le 1$
\begin{align}
			\limsup_{n\to\infty} \frac{\ln(t_{\textnormal{mix}}^\omega(n))}{\ln n}\ \le\ \frac1\kappa. \label{CT3.0.0.0.1}
\end{align}
Further, for the lower bound of the mixing time, we get for any constant $c>0$
\begin{align}
&	d_n\left(\left\lfloor c n^{\frac1\kappa}(\ln n)^{-\frac4\kappa} \right\rfloor -2\right) \nonumber\\
\ge\ & \pi_{\omega^n}\left([n- 2(\ln n)^2,n]\right)	- P_{\omega^n}\left(Y_{\left\lfloor c n^{\frac1\kappa}(\ln n)^{-\frac4\kappa} \right\rfloor -2} \ge n- 2(\ln n)^2\right) \nonumber\\
\ge\ & \pi_{\omega^n}\left([n- 2(\ln n)^2,n]\right)	- P_{\omega^n}\left(T^Y_{\lfloor n- 2(\ln n)^2\rfloor} \le \left\lfloor c n^{\frac1\kappa}(\ln n)^{-\frac4\kappa} \right\rfloor -2\right) \nonumber\\
 \ge\ & \pi_{\omega^n}\left([n- 2(\ln n)^2,n]\right)	- P_{\omega^n}\left(T_{\lfloor n- 2(\ln n)^2\rfloor} < \left\lfloor c n^{\frac1\kappa}(\ln n)^{-\frac4\kappa} \right\rfloor -1\right). \label{CT3.0.0.1}
\end{align}
To reach position $\lfloor n- 2(\ln n)^2\rfloor$ the lazy RWRE has first of all to cross the highest block on the interval $[0,\lfloor n- 2(\ln n)^2\rfloor]$. For environments $\omega\in B_4(\lfloor n- 2(\ln n)^2\rfloor)$, we have at least one block with a height of more than $\frac1\kappa\left(\ln \left(\frac{n}2\right)- 4 \ln\ln n\right)$. Now, using Proposition 4.2 in \cite{FGP}, we get for \eqref{CT3.0.0.1} for environments $\omega\in B_4(\lfloor n- 2(\ln n)^2\rfloor)$ and $c:= \frac{1}{\gamma\cdot 2^{1+\frac1\kappa}}$,
\begin{align*}
		d_n\left(\left\lfloor c n^{\frac1\kappa}(\ln n)^{-\frac4\kappa} \right\rfloor -2\right)
	\ge\ & \pi_{\omega^n}\left([n- 2(\ln n)^2,n]\right)	- \gamma \cdot c\cdot  2^{\frac1\kappa} \
	=\  \pi_{\omega^n}\left([n- 2(\ln n)^2,n]\right) - \frac12.
\end{align*}
Finally, using Lemma \ref{PL5} and Lemma \ref{CL1} yield for $\textbf{P}$-almost every environment $\omega$ and large $n$
\begin{align*}
	t_{\text{mix}}^{\omega} (n)\ \ge\ \left\lfloor \frac1{\gamma \cdot 2^{1+\frac1\kappa}} n^{\frac1\kappa}(\ln n)^{-\frac4\kappa} \right\rfloor -2
\end{align*}
and therefore 
\begin{align*}
	\liminf_{n\to\infty} \frac{\ln t_{\text{mix}}^{\omega} (n)}{ \ln n} \ \ge\ \frac1\kappa,
\end{align*}
which together with \eqref{CT3.0.0.0.1} shows that for $\kappa \le 1$ and $\textbf{P}$-almost every environment $\omega$ we have 
\begin{align*}
	\lim_{n\to\infty} \frac{\ln t_{\text{mix}}^{\omega} (n)}{ \ln n} \ =\ \frac1\kappa.
\end{align*}
In Theorem \ref{CT2} we show that a sequence of lazy RWRE exhibits a cutoff for $\kappa >1$. Therefore, the mixing time has the same order as the cutoff times, and due to equations  \eqref{Ergodensatz}, \eqref{PL0.1.0.01} and Lemma \ref{PL0.1} we have for $\textbf{P}$-almost every environment $\omega$ 
\begin{align*}
	\lim_{n\to\infty} \frac1n \cdot  t_{\text{mix}}^{\omega} (n)\ =\ 2 \mathbb{E} T_1.
\end{align*}  

\end{proof} 
\noindent
{\bf Acknowledgements:} We are very much indepted to Jonathon Peterson for helpful comments on 
the thesis of the second author. We also thank
the referee for his careful lecture, resulting in many improvements.

\bibliographystyle{alpha}

\medskip

{\footnotesize

Nina Gantert: Technische Universit\"at M\"unchen,
Fakult\"at f\"ur Mathematik,
Boltzmannstra\ss e~3, 
85748~Garching bei M\"unchen.
Germany.\\
gantert@ma.tum.de\\
http://www-m14.ma.tum.de/en/people/gantert/   \\

Thomas Kochler: Technische Universit\"at M\"unchen,
Fakult\"at f\"ur Mathematik,
Boltzmannstra\ss e~3, 
85748~Garching bei M\"unchen.
Germany.\\  }

\end{document}